\documentclass{amsart}

\usepackage{wrapfig}
\usepackage{tabu}
\usepackage{longtable}
\usepackage{amsmath,amsthm,amssymb,comment,fullpage}
\usepackage{subcaption}
\usepackage[T1]{fontenc}
\usepackage{mathrsfs}
\usepackage{latexsym}
\usepackage[dvips]{graphics}
\usepackage{epsfig}
\usepackage{floatflt}
\usepackage{amsmath,amsfonts,amsthm,amssymb,amscd}
\usepackage{color}
\usepackage{graphicx}

\usepackage{hyperref}
\usepackage{url}
\usepackage{breakurl}
\usepackage{comment}
\usepackage{algorithm}
\usepackage{algpseudocode}
\usepackage{fancybox}
\newcommand{\bburl}[1]{\textcolor{blue}{\url{#1}}}

\numberwithin{equation}{section}

\newtheorem{thm}{Theorem}[section]

\newtheorem{cor}[thm]{Corollary}

\theoremstyle{plain}

\newtheorem{corollary}[thm]{Corollary}

\newtheorem{lemma}[thm]{Lemma}
\newtheorem{proposition}[thm]{Proposition}
\newtheorem{theorem}[thm]{Theorem}

\theoremstyle{remark}
\newtheorem{rem}[thm]{Remark}
\newtheorem{remark}[thm]{Remark}

\theoremstyle{definition}
\newtheorem{defi}[thm]{Definition}

\newtheorem{definition}[thm]{Definition}

\newcommand\be{\begin{equation}}
\newcommand\ee{\end{equation}}
\newcommand\bea{\begin{eqnarray}}
\newcommand\eea{\end{eqnarray}}
\newcommand\bi{\begin{itemize}}
\newcommand\ei{\end{itemize}}
\newcommand\ben{\begin{enumerate}}
\newcommand\een{\end{enumerate}}
\newcommand\bc{\begin{center}}
\newcommand\ec{\end{center}}
\newcommand\ba{\begin{array}}
\newcommand\ea{\end{array}}

\usepackage{pgf,tikz}
\usepackage{mathrsfs}
\usetikzlibrary{arrows}

\definecolor{ffqqqq}{rgb}{1.,0.,0.}
\definecolor{qqffqq}{rgb}{0.,1.,0.}
\definecolor{qqqqff}{rgb}{0.,0.,1.}
\definecolor{cqcqcq}{rgb}{0.7529411764705882,0.7529411764705882,0.7529411764705882}

\newcommand{\R}{\ensuremath{\mathbb{R}}}

\newcommand{\N}{\mathbb{N}}

\usepackage{multicol}

\newcommand{\hr}[1]{\href{#1}{\url{#1}}}

\title{Classification of Crescent Configurations}

\author{Yuda Chen}
\email{\textcolor{blue}{\href{mailto:math.yuda.chen@gmail.com}{math.yuda.chen@gmail.com}}}
\address{School of Mathematical Sciences and LPMC, Nankai University, Tianjin 300071, China}

\author{Rebecca F. Durst}
\email{\textcolor{blue}{\href{mailto:rebecca_durst@brown.edu}{rebecca\_durst@brown.edu}}}
\address{Division of Applied Mathematics, Brown University, Providence, RI 02912}

\author{Chi N.Y. Huynh}
\email{\textcolor{blue}{\href{mailto:nyhuynh2@illinois.edu}{nyhuynh2@illinois.edu}}}
\address{Department of Mathematics, University of Illinois at Urbana-Champaign, Urbana, IL 61801}

\author{Steven J. Miller}
\email{\textcolor{blue}{\href{mailto:sjm1@williams.edu}{sjm1@williams.edu}} and \textcolor{blue}{\href{mailto:stevenm1@andrew.cmu.edu}{stevenm1@andrew.cmu.edu}}}
\address{Department of Mathematics and Statistics, Williams College, Williamstown, MA 01267}
\curraddr{Department of Mathematical Sciences, Carnegie Mellon University, Pittsburgh, PA 15213}

\author{Eyvindur A. Palsson}
\email{\textcolor{blue}{\href{mailto:palsson@vt.edu} {palsson@vt.edu}}}
\address{Department of Mathematics, Virginia Tech University, Blacksburg, VA 24061}

\author{Yifei Wang}
\email{\textcolor{blue}{\href{mailto:yifeiwang.math@gmail.com}{yifeiwang.math@gmail.com}}}
\address{School of Mathematical Sciences and LPMC, Nankai University, Tianjin 300071, China}

\thanks{This work was supported financially by Williams College Finnerty Fund, Simons Foundation Grant \#360560, NSF Grants DMS1265673, DMS1561945 and NSF Grant DMS1347804. The authors want to particularly thank Max Hlavacek, an early contributor to this project, for all her contributions. We want to thank Williams College and SMALL REU for continuous support, Kevin Kwan for helpful conversations on geometrical properties of crescent configurations, Ileana Streinu for insightful advice on rigidity, and Ferenc Sz{\"o}ll{\H{o}}si for directing us towards Liu's work and pointing out the need to define strict graph isomorphism classes. Some of the figures and tables were generated with the help of GPT-6. We also want to thank the referee for detailed advice.}

\subjclass[2010]{52C10 (primary) 52C35 (secondary)}

\keywords{}

\date{\today}

\begin{document}
	
\begin{abstract}
Let $n$ points be in crescent configuration in $\mathbb{R}^d$ if they lie in general position in $\mathbb{R}^d$ and determine $n-1$ distinct distances, such that for every $1 \leq i \leq n-1$ there is a distance that occurs exactly $i$ times. For the planar problem, Erd{\H{o}}s conjectured in \cite{Erd89} that there exists \(N\) such that no crescent configuration exists on \(n\) points whenever \(n\geq N\). In the same paper, he also gave constructions for $n$ up to 8, with $n=5$ due to Pomerance and $n=7,8$ due to Pal{\'a}sti, but nothing is yet known for $n \geq 9$. Most recently, Burt et al. \cite{SM15} proved that a crescent configuration on $n$ points exists in $\mathbb{R}^{n-2}$ for $n \geq 3$. In this paper, we classify planar crescent configurations by their complete distance-equality patterns, up to relabeling the vertices. We restrict attention to the strict distance classes defined below. There are three such classes on four vertices, all of which admit planar crescent realizations. On five vertices, there are \(49\) strict classes, exactly \(31\) of which admit planar crescent realizations. On six vertices, there are \(6837\) strict classes, exactly six of which admit planar crescent realizations. The classification combines finite enumeration with exact algebraic verification of realizations and nonexistence. The distance-geometric criteria also apply in higher dimensions.
\end{abstract}

\maketitle
\tableofcontents

\section{Introduction}\label{intro}
Erd{\H{o}}s once wrote, ``my most striking contribution to geometry is, no doubt, my problem on the number of distinct distances'' \cite{Erd96}. The referred question, which asks what is the minimum number of distinct distances determined by $n$ points, was first asked in 1946 \cite{Erd46} and marked the beginning of a chain of variants. See \cite{She} and \cite{SM15} for a survey on these. Although one would expect all distances between $n$ points to be different if they were to be placed in the plane at random, if the distances are regularly placed, such as on a lattice, then many distances may repeat. Erd{\H{o}}s's conjectured lower bound, $\Omega(n/\sqrt{\log{n}})$, attained by a $\sqrt{n}\times \sqrt{n}$ integer lattice, was essentially proven up to a $\sqrt{\log{n}}$ factor by Guth and Katz \cite{GK}.

The variant we study in this paper is one where the distances have prescribed multiplicities. One says $n$ points are in \emph{crescent configuration} in $\mathbb{R}^d$ if they are in general position and determine $n-1$ distinct distances such that for every $1 \leq i \leq n-1$, there is a distance that occurs exactly $i$ times. The notion of \emph{general position}, which in $\mathbb{R}^2$ means no $3$ points lying in the same line and no $4$ points lying in the same circle, is critical to the non-trivial construction of crescent configurations (for the definition of general position in $\mathbb{R}^d$ for $d >2$, see \cite{SM15}). Erd{\H{o}}s conjectured that there exists a sufficiently large $N$ such that no planar crescent configuration exists on $N$ or more points \cite{Erd89}. Although constructions have been provided for $n \leq 8$ by Erd{\H{o}}s in \cite{Erd89}, with the case $n = 5$ due to C. Pomerance and $n = 7,8$ due to I. Pal{\'a}sti \cite{Pal87,Pal89}, little progress has been made towards a construction for $n \geq 9$. One problem often encountered in the search for these configurations is the lack of understanding of their properties, and the difficulty in exhibiting the configurations' information combinatorially.

Later, A. Liu published a manuscript in which he lists known crescent configurations, and begins classifying crescent configurations (which he refers to as complete systems on $n$ points) based on their labeled graph structures \cite{Liu}.  In particular, he identified 3 combinatorial types of crescent configurations on 4 points.


In this paper, we continue this approach to studying these crescent configurations, working to formalize Liu's idea of classifying these configurations using combinatorial data.  Along the way, we borrow techniques from distance geometry and graph theory. In  our methods, we employ a concept of \textit{strict distance classes}, which are distance classes  of crescent configurations which are more likely to obey  general position. The concepts of distance class and strict distance class will be defined rigorously in Sections 2 and 3. Our main results are obtained using three algorithms: one enumerates strict distance classes, one verifies proposed exact planar realizations, and one attempts to exclude classes without a verified realization. The first two apply to every \(n\geq4\); the exclusion algorithm is presented for five points and extended to six points in Section~\ref{sec:six-point}.

\begin{theorem}\label{thm:3on4}
There are exactly three strict distance classes on four vertices, and each admits a crescent realization in \(\mathbb R^2\). Representatives are shown in Figure~\ref{fig:mcr}.
\end{theorem}


\begin{figure}[h]
\begin{tikzpicture}[scale=0.7]

\draw[red] (-1,1) -- (0,2.732);
\draw[red] (0,2.732) -- (1,1);
\draw[red] (-1,1) -- (1,1);
\draw[blue] (-1,1) -- (0,-2.873);
\draw[blue] (1,1) -- (0,-2.873);
\draw[green] (0,2.732) -- (0,-2.873);

\draw[red] (2,2) -- (3,0.268);
\draw[red] (3,0.268) -- (4,2);
\draw[red] (2,2) -- (4,2);
\draw[blue] (2,2) -- (3,-1.873);
\draw[blue] (4,2) -- (3,-1.873);
\draw[green] (3,0.268) -- (3,-1.873);

\node at (1.5, -4) {M};

\draw[red] (8,-1) -- (8,1);
\draw[red] (8,-1) -- (7.25,-2.854);
\draw[red] (8,-1) -- (8.75,-2.854);
\draw[blue] (8,1) -- (7.25,-2.854);
\draw[blue] (8,1) -- (8.75,-2.854);
\draw[green] (7.25,-2.854) -- (8.75,-2.854);

\node at (8, -4) {C};

\draw[red] (13,-1) -- (16,-1);
\draw[red] (13,-1) -- (15.484,0.683);
\draw[red] (16,-1) -- (13.75,0.984);
\draw[blue] (13.75,0.984) -- (15.484,0.683);
\draw[blue] (15.484,0.683) -- (16,-1);
\draw[green] (13,-1) -- (13.75,0.984);

\node at (14.5, -4) {R};

\end{tikzpicture}
\caption{Types M, C, and R.}
\label{fig:mcr}
\end{figure}
\begin{theorem}\label{thm:27on5}
There are exactly \(49\) strict distance classes on five vertices. Exactly \(31\) of them admit crescent realizations in \(\mathbb R^2\). Representatives of these \(31\) classes are shown in Figure~\ref{fig:allFIVE}.
\end{theorem}

\begin{figure}[htbp]
\centering
\resizebox{\linewidth}{!}{%
\begin{tikzpicture}[
    x=1cm,y=1cm,
    line width=0.32pt,
    line cap=round,
    line join=round,
    font=\small
]

\begin{scope}[shift={(1.743429,-1.248000)}]
  \coordinate (A) at (0.400571,0.396584);
  \coordinate (B) at (-0.400571,0.159656);
  \coordinate (C) at (-0.107145,0.142260);
  \coordinate (D) at (-0.268924,-0.103157);
  \coordinate (E) at (-0.286320,-0.396584);
  \draw[green!50!black] (A)--(E);
  \draw[blue] (A)--(B) (A)--(D);
  \draw[red] (A)--(C) (B)--(E) (C)--(E);
  \draw[orange] (B)--(C) (B)--(D) (C)--(D) (D)--(E);
  \foreach \P in {A,B,C,D,E}
    \filldraw[fill=black!73,draw=black,line width=0.24pt]
      (\P) circle[radius=1.20pt];
  \node[inner sep=0pt] at (-0.031429,-0.592000) {(1)};
\end{scope}

\begin{scope}[shift={(2.824571,-1.144000)}]
  \coordinate (A) at (0.178017,0.371084);
  \coordinate (B) at (0.449669,0.104088);
  \coordinate (C) at (-0.184378,-0.269324);
  \coordinate (D) at (0.551429,-0.262963);
  \coordinate (E) at (-0.551429,-0.371084);
  \draw[green!50!black] (A)--(E);
  \draw[blue] (B)--(E) (D)--(E);
  \draw[red] (A)--(B) (B)--(D) (C)--(E);
  \draw[orange] (A)--(C) (A)--(D) (B)--(C) (C)--(D);
  \foreach \P in {A,B,C,D,E}
    \filldraw[fill=black!73,draw=black,line width=0.24pt]
      (\P) circle[radius=1.20pt];
  \node[inner sep=0pt] at (0.071429,-0.776000) {(2)};
\end{scope}

\begin{scope}[shift={(4.384000,-1.160000)}]
  \coordinate (A) at (-0.277066,0.420300);
  \coordinate (B) at (0.045762,0.156452);
  \coordinate (C) at (-0.344151,0.008800);
  \coordinate (D) at (-0.608000,-0.314028);
  \coordinate (E) at (0.608000,-0.420300);
  \draw[green!50!black] (C)--(E);
  \draw[blue] (A)--(E) (D)--(E);
  \draw[red] (A)--(D) (B)--(D) (B)--(E);
  \draw[orange] (A)--(B) (A)--(C) (B)--(C) (C)--(D);
  \foreach \P in {A,B,C,D,E}
    \filldraw[fill=black!73,draw=black,line width=0.24pt]
      (\P) circle[radius=1.20pt];
  \node[inner sep=0pt] at (-0.112000,-0.808000) {(3)};
\end{scope}

\begin{scope}[shift={(6.550857,-1.096571)}]
  \coordinate (A) at (0.168955,0.789101);
  \coordinate (B) at (0.653755,0.233803);
  \coordinate (C) at (0.056318,-0.198003);
  \coordinate (D) at (-0.854857,-0.594009);
  \coordinate (E) at (0.854857,-0.789101);
  \draw[green!50!black] (B)--(E);
  \draw[blue] (A)--(B) (B)--(C);
  \draw[red] (A)--(C) (C)--(D) (C)--(E);
  \draw[orange] (A)--(D) (A)--(E) (B)--(D) (D)--(E);
  \foreach \P in {A,B,C,D,E}
    \filldraw[fill=black!73,draw=black,line width=0.24pt]
      (\P) circle[radius=1.20pt];
  \node[inner sep=0pt] at (-0.038857,-0.967429) {(4)};
\end{scope}

\begin{scope}[shift={(8.807429,-1.016571)}]
  \coordinate (A) at (-0.053953,0.424571);
  \coordinate (B) at (0.474305,0.419157);
  \coordinate (C) at (0.632397,-0.084920);
  \coordinate (D) at (-0.632397,-0.204763);
  \coordinate (E) at (-0.152011,-0.424571);
  \draw[green!50!black] (B)--(E);
  \draw[blue] (B)--(D) (C)--(D);
  \draw[red] (A)--(B) (B)--(C) (D)--(E);
  \draw[orange] (A)--(C) (A)--(D) (A)--(E) (C)--(E);
  \foreach \P in {A,B,C,D,E}
    \filldraw[fill=black!73,draw=black,line width=0.24pt]
      (\P) circle[radius=1.20pt];
  \node[inner sep=0pt] at (0.024571,-0.759429) {(5)};
\end{scope}

\begin{scope}[shift={(11.328000,-1.089333)}]
  \coordinate (A) at (-0.493770,0.939322);
  \coordinate (B) at (1.018311,0.590295);
  \coordinate (C) at (-1.280000,-0.590096);
  \coordinate (D) at (-0.232081,-0.590295);
  \coordinate (E) at (1.280000,-0.939322);
  \draw[green!50!black] (C)--(D);
  \draw[blue] (A)--(C) (B)--(D);
  \draw[red] (A)--(E) (B)--(C) (C)--(E);
  \draw[orange] (A)--(B) (A)--(D) (B)--(E) (D)--(E);
  \foreach \P in {A,B,C,D,E}
    \filldraw[fill=black!73,draw=black,line width=0.24pt]
      (\P) circle[radius=1.20pt];
  \node[inner sep=0pt] at (-1.216000,-0.894667) {(6)};
\end{scope}

\begin{scope}[shift={(13.872000,-1.088000)}]
  \coordinate (A) at (0.686931,0.672000);
  \coordinate (B) at (-0.686931,0.363597);
  \coordinate (C) at (0.166071,0.367828);
  \coordinate (D) at (-0.256766,-0.373009);
  \coordinate (E) at (0.267085,-0.672000);
  \draw[green!50!black] (C)--(E);
  \draw[blue] (A)--(C) (D)--(E);
  \draw[red] (B)--(C) (B)--(D) (C)--(D);
  \draw[orange] (A)--(B) (A)--(D) (A)--(E) (B)--(E);
  \foreach \P in {A,B,C,D,E}
    \filldraw[fill=black!73,draw=black,line width=0.24pt]
      (\P) circle[radius=1.20pt];
  \node[inner sep=0pt] at (-0.784000,-0.672000) {(7)};
\end{scope}

\begin{scope}[shift={(1.392000,-2.696000)}]
  \coordinate (A) at (-0.086456,0.359169);
  \coordinate (B) at (0.432000,0.230453);
  \coordinate (C) at (0.146549,0.189185);
  \coordinate (D) at (0.362871,-0.299250);
  \coordinate (E) at (-0.432000,-0.359169);
  \draw[green!50!black] (B)--(E);
  \draw[blue] (A)--(C) (B)--(C);
  \draw[red] (A)--(B) (B)--(D) (C)--(D);
  \draw[orange] (A)--(D) (A)--(E) (C)--(E) (D)--(E);
  \foreach \P in {A,B,C,D,E}
    \filldraw[fill=black!73,draw=black,line width=0.24pt]
      (\P) circle[radius=1.20pt];
  \node[inner sep=0pt] at (-0.112000,-0.888000) {(8)};
\end{scope}

\begin{scope}[shift={(2.777455,-2.816000)}]
  \coordinate (A) at (0.122947,0.368000);
  \coordinate (B) at (0.495914,0.287928);
  \coordinate (C) at (0.042876,-0.004966);
  \coordinate (D) at (-0.495914,-0.032101);
  \coordinate (E) at (0.160015,-0.368000);
  \draw[green!50!black] (B)--(D);
  \draw[blue] (B)--(C) (C)--(D);
  \draw[red] (A)--(B) (A)--(C) (C)--(E);
  \draw[orange] (A)--(D) (A)--(E) (B)--(E) (D)--(E);
  \foreach \P in {A,B,C,D,E}
    \filldraw[fill=black!73,draw=black,line width=0.24pt]
      (\P) circle[radius=1.20pt];
  \node[inner sep=0pt] at (-0.041455,-0.560000) {(9)};
\end{scope}

\begin{scope}[shift={(4.577846,-2.744000)}]
  \coordinate (A) at (-0.449846,0.325184);
  \coordinate (B) at (0.121376,0.267698);
  \coordinate (C) at (-0.188911,0.051242);
  \coordinate (D) at (0.449846,-0.203160);
  \coordinate (E) at (-0.226794,-0.325184);
  \draw[green!50!black] (A)--(D);
  \draw[blue] (A)--(B) (B)--(D);
  \draw[red] (A)--(C) (B)--(C) (C)--(E);
  \draw[orange] (A)--(E) (B)--(E) (C)--(D) (D)--(E);
  \foreach \P in {A,B,C,D,E}
    \filldraw[fill=black!73,draw=black,line width=0.24pt]
      (\P) circle[radius=1.20pt];
  \node[inner sep=0pt] at (-0.273846,-0.712000) {(10)};
\end{scope}

\begin{scope}[shift={(5.888000,-2.688000)}]
  \coordinate (A) at (0.416000,0.316372);
  \coordinate (B) at (-0.304332,0.084727);
  \coordinate (C) at (0.109050,0.035068);
  \coordinate (D) at (-0.416000,-0.316372);
  \coordinate (E) at (0.340643,-0.310930);
  \draw[green!50!black] (A)--(D);
  \draw[blue] (A)--(E) (C)--(D);
  \draw[red] (A)--(B) (B)--(E) (D)--(E);
  \draw[orange] (A)--(C) (B)--(C) (B)--(D) (C)--(E);
  \foreach \P in {A,B,C,D,E}
    \filldraw[fill=black!73,draw=black,line width=0.24pt]
      (\P) circle[radius=1.20pt];
  \node[inner sep=0pt] at (-0.064000,-0.800000) {(11)};
\end{scope}

\begin{scope}[shift={(7.592000,-2.721455)}]
  \coordinate (A) at (-0.053945,0.449455);
  \coordinate (B) at (-0.572453,0.350505);
  \coordinate (C) at (-0.227506,-0.049061);
  \coordinate (D) at (0.572453,-0.048100);
  \coordinate (E) at (-0.571492,-0.449455);
  \draw[green!50!black] (A)--(E);
  \draw[blue] (B)--(D) (D)--(E);
  \draw[red] (A)--(D) (B)--(E) (C)--(D);
  \draw[orange] (A)--(B) (A)--(C) (B)--(C) (C)--(E);
  \foreach \P in {A,B,C,D,E}
    \filldraw[fill=black!73,draw=black,line width=0.24pt]
      (\P) circle[radius=1.20pt];
  \node[inner sep=0pt] at (-0.232000,-0.734545) {(12)};
\end{scope}

\begin{scope}[shift={(9.032571,-2.702286)}]
  \coordinate (A) at (0.420019,0.337714);
  \coordinate (B) at (-0.420019,0.188268);
  \coordinate (C) at (0.408179,-0.016880);
  \coordinate (D) at (0.054289,-0.042170);
  \coordinate (E) at (-0.382431,-0.337714);
  \draw[green!50!black] (A)--(E);
  \draw[blue] (A)--(C) (C)--(D);
  \draw[red] (A)--(B) (B)--(C) (C)--(E);
  \draw[orange] (A)--(D) (B)--(D) (B)--(E) (D)--(E);
  \foreach \P in {A,B,C,D,E}
    \filldraw[fill=black!73,draw=black,line width=0.24pt]
      (\P) circle[radius=1.20pt];
  \node[inner sep=0pt] at (-0.088571,-0.737714) {(13)};
\end{scope}

\begin{scope}[shift={(10.504000,-2.701714)}]
  \coordinate (A) at (0.361211,0.465143);
  \coordinate (B) at (0.438476,0.212856);
  \coordinate (C) at (-0.186189,0.159522);
  \coordinate (D) at (-0.438476,0.082257);
  \coordinate (E) at (-0.132855,-0.465143);
  \draw[green!50!black] (A)--(E);
  \draw[blue] (A)--(B) (C)--(D);
  \draw[red] (A)--(D) (B)--(D) (B)--(E);
  \draw[orange] (A)--(C) (B)--(C) (C)--(E) (D)--(E);
  \foreach \P in {A,B,C,D,E}
    \filldraw[fill=black!73,draw=black,line width=0.24pt]
      (\P) circle[radius=1.20pt];
  \node[inner sep=0pt] at (0.216000,-0.690286) {(14)};
\end{scope}

\begin{scope}[shift={(12.864000,-3.784000)}]
  \coordinate (A) at (-1.495741,1.368000);
  \coordinate (B) at (1.495741,1.077397);
  \coordinate (C) at (-0.375794,0.226402);
  \coordinate (D) at (1.215940,0.071776);
  \coordinate (E) at (-0.251669,-1.368000);
  \draw[green!50!black] (B)--(D);
  \draw[blue] (B)--(C) (D)--(E);
  \draw[red] (A)--(C) (C)--(D) (C)--(E);
  \draw[orange] (A)--(B) (A)--(D) (A)--(E) (B)--(E);
  \foreach \P in {A,B,C,D,E}
    \filldraw[fill=black!73,draw=black,line width=0.24pt]
      (\P) circle[radius=1.20pt];
  \node[inner sep=0pt] at (0.176000,-1.576000) {(15)};
\end{scope}

\begin{scope}[shift={(1.624000,-4.224000)}]
  \coordinate (A) at (0.383342,0.507414);
  \coordinate (B) at (0.082886,0.065844);
  \coordinate (C) at (0.615525,0.026427);
  \coordinate (D) at (-0.632000,-0.279248);
  \coordinate (E) at (0.632000,-0.507414);
  \draw[green!50!black] (A)--(E);
  \draw[blue] (B)--(D) (B)--(E);
  \draw[red] (A)--(D) (C)--(D) (D)--(E);
  \draw[orange] (A)--(B) (A)--(C) (B)--(C) (C)--(E);
  \foreach \P in {A,B,C,D,E}
    \filldraw[fill=black!73,draw=black,line width=0.24pt]
      (\P) circle[radius=1.20pt];
  \node[inner sep=0pt] at (-0.456000,-0.992000) {(16)};
\end{scope}

\begin{scope}[shift={(3.840000,-5.480000)}]
  \coordinate (A) at (1.389509,1.720000);
  \coordinate (B) at (-0.469566,1.336425);
  \coordinate (C) at (-1.389509,-0.323996);
  \coordinate (D) at (-0.371853,-0.559292);
  \coordinate (E) at (1.130162,-1.720000);
  \draw[green!50!black] (C)--(D);
  \draw[blue] (A)--(D) (C)--(E);
  \draw[red] (A)--(C) (A)--(E) (B)--(E);
  \draw[orange] (A)--(B) (B)--(C) (B)--(D) (D)--(E);
  \foreach \P in {A,B,C,D,E}
    \filldraw[fill=black!73,draw=black,line width=0.24pt]
      (\P) circle[radius=1.20pt];
  \node[inner sep=0pt] at (-0.496000,-1.336000) {(17)};
\end{scope}

\begin{scope}[shift={(6.178667,-4.425333)}]
  \coordinate (A) at (-0.415486,0.409333);
  \coordinate (B) at (0.415486,0.402000);
  \coordinate (C) at (-0.000160,0.387483);
  \coordinate (D) at (-0.234851,-0.115330);
  \coordinate (E) at (0.235746,-0.409333);
  \draw[green!50!black] (A)--(E);
  \draw[blue] (A)--(C) (B)--(C);
  \draw[red] (A)--(D) (C)--(D) (D)--(E);
  \draw[orange] (A)--(B) (B)--(D) (B)--(E) (C)--(E);
  \foreach \P in {A,B,C,D,E}
    \filldraw[fill=black!73,draw=black,line width=0.24pt]
      (\P) circle[radius=1.20pt];
  \node[inner sep=0pt] at (-0.242667,-0.550667) {(18)};
\end{scope}

\begin{scope}[shift={(7.608000,-4.190857)}]
  \coordinate (A) at (-0.038239,0.465143);
  \coordinate (B) at (0.566633,0.188692);
  \coordinate (C) at (-0.566633,0.061291);
  \coordinate (D) at (0.087202,-0.060347);
  \coordinate (E) at (0.444995,-0.465143);
  \draw[green!50!black] (A)--(E);
  \draw[blue] (B)--(C) (C)--(E);
  \draw[red] (A)--(D) (B)--(D) (D)--(E);
  \draw[orange] (A)--(B) (A)--(C) (B)--(E) (C)--(D);
  \foreach \P in {A,B,C,D,E}
    \filldraw[fill=black!73,draw=black,line width=0.24pt]
      (\P) circle[radius=1.20pt];
  \node[inner sep=0pt] at (-0.376000,-0.353143) {(19)};
\end{scope}

\begin{scope}[shift={(9.719429,-4.320000)}]
  \coordinate (A) at (0.699295,0.672000);
  \coordinate (B) at (-0.686010,0.540034);
  \coordinate (C) at (0.026214,-0.229089);
  \coordinate (D) at (-1.060055,-0.520668);
  \coordinate (E) at (1.060055,-0.672000);
  \draw[green!50!black] (B)--(C);
  \draw[blue] (A)--(B) (A)--(E);
  \draw[red] (A)--(D) (B)--(E) (D)--(E);
  \draw[orange] (A)--(C) (B)--(D) (C)--(D) (C)--(E);
  \foreach \P in {A,B,C,D,E}
    \filldraw[fill=black!73,draw=black,line width=0.24pt]
      (\P) circle[radius=1.20pt];
  \node[inner sep=0pt] at (1.400571,-0.256000) {(20)};
\end{scope}

\begin{scope}[shift={(1.215429,-6.327619)}]
  \coordinate (A) at (-0.757450,0.980952);
  \coordinate (B) at (0.236602,0.636714);
  \coordinate (C) at (0.880223,-0.106324);
  \coordinate (D) at (-0.336495,-0.161981);
  \coordinate (E) at (-0.880223,-0.980952);
  \draw[green!50!black] (A)--(B);
  \draw[blue] (A)--(D) (C)--(D);
  \draw[red] (B)--(C) (B)--(D) (D)--(E);
  \draw[orange] (A)--(C) (A)--(E) (B)--(E) (C)--(E);
  \foreach \P in {A,B,C,D,E}
    \filldraw[fill=black!73,draw=black,line width=0.24pt]
      (\P) circle[radius=1.20pt];
  \node[inner sep=0pt] at (-0.079429,-1.256381) {(21)};
\end{scope}

\begin{scope}[shift={(6.161333,-5.838545)}]
  \coordinate (A) at (-0.289191,0.510545);
  \coordinate (B) at (0.115974,0.293217);
  \coordinate (C) at (-0.453282,0.081053);
  \coordinate (D) at (0.453282,-0.019214);
  \coordinate (E) at (-0.315161,-0.510545);
  \draw[green!50!black] (A)--(E);
  \draw[blue] (B)--(C) (C)--(E);
  \draw[red] (A)--(B) (A)--(C) (B)--(D);
  \draw[orange] (A)--(D) (B)--(E) (C)--(D) (D)--(E);
  \foreach \P in {A,B,C,D,E}
    \filldraw[fill=black!73,draw=black,line width=0.24pt]
      (\P) circle[radius=1.20pt];
  \node[inner sep=0pt] at (-0.449333,-0.737455) {(22)};
\end{scope}

\begin{scope}[shift={(7.656000,-5.729846)}]
  \coordinate (A) at (0.440000,0.433040);
  \coordinate (B) at (-0.440000,0.347928);
  \coordinate (C) at (-0.177264,0.334008);
  \coordinate (D) at (0.426079,0.170303);
  \coordinate (E) at (0.262375,-0.433040);
  \draw[green!50!black] (B)--(E);
  \draw[blue] (A)--(D) (B)--(C);
  \draw[red] (A)--(C) (C)--(D) (D)--(E);
  \draw[orange] (A)--(B) (A)--(E) (B)--(D) (C)--(E);
  \foreach \P in {A,B,C,D,E}
    \filldraw[fill=black!73,draw=black,line width=0.24pt]
      (\P) circle[radius=1.20pt];
  \node[inner sep=0pt] at (-0.344000,-0.478154) {(23)};
\end{scope}

\begin{scope}[shift={(9.320000,-5.704000)}]
  \coordinate (A) at (-0.280000,0.464915);
  \coordinate (B) at (0.280000,0.161452);
  \coordinate (C) at (-0.008916,0.023356);
  \coordinate (D) at (-0.262807,-0.171791);
  \coordinate (E) at (0.164439,-0.464915);
  \draw[green!50!black] (A)--(E);
  \draw[blue] (B)--(C) (C)--(D);
  \draw[red] (A)--(C) (C)--(E) (D)--(E);
  \draw[orange] (A)--(B) (A)--(D) (B)--(D) (B)--(E);
  \foreach \P in {A,B,C,D,E}
    \filldraw[fill=black!73,draw=black,line width=0.24pt]
      (\P) circle[radius=1.20pt];
  \node[inner sep=0pt] at (0.008000,-0.760000) {(24)};
\end{scope}

\begin{scope}[shift={(11.000000,-5.760000)}]
  \coordinate (A) at (-0.456000,0.478427);
  \coordinate (B) at (-0.053728,0.433573);
  \coordinate (C) at (0.456000,0.173100);
  \coordinate (D) at (-0.098583,0.031301);
  \coordinate (E) at (-0.359055,-0.478427);
  \draw[green!50!black] (C)--(E);
  \draw[blue] (A)--(B) (B)--(D);
  \draw[red] (A)--(C) (A)--(E) (B)--(E);
  \draw[orange] (A)--(D) (B)--(C) (C)--(D) (D)--(E);
  \foreach \P in {A,B,C,D,E}
    \filldraw[fill=black!73,draw=black,line width=0.24pt]
      (\P) circle[radius=1.20pt];
  \node[inner sep=0pt] at (-0.184000,-0.672000) {(25)};
\end{scope}

\begin{scope}[shift={(12.486286,-6.088571)}]
  \coordinate (A) at (0.438783,0.503334);
  \coordinate (B) at (-0.407023,0.414465);
  \coordinate (C) at (-0.061976,-0.184071);
  \coordinate (D) at (-0.726286,-0.373797);
  \coordinate (E) at (0.726286,-0.503334);
  \draw[green!50!black] (A)--(E);
  \draw[blue] (B)--(C) (C)--(D);
  \draw[red] (A)--(D) (B)--(E) (D)--(E);
  \draw[orange] (A)--(B) (A)--(C) (B)--(D) (C)--(E);
  \foreach \P in {A,B,C,D,E}
    \filldraw[fill=black!73,draw=black,line width=0.24pt]
      (\P) circle[radius=1.20pt];
  \node[inner sep=0pt] at (1.033714,0.328571) {(26)};
\end{scope}

\begin{scope}[shift={(14.536000,-5.944000)}]
  \coordinate (A) at (0.160151,0.632000);
  \coordinate (B) at (0.611839,0.202431);
  \coordinate (C) at (-0.269418,0.180312);
  \coordinate (D) at (-0.034473,-0.397055);
  \coordinate (E) at (-0.611839,-0.632000);
  \draw[green!50!black] (A)--(D);
  \draw[blue] (A)--(E) (B)--(E);
  \draw[red] (B)--(C) (B)--(D) (C)--(E);
  \draw[orange] (A)--(B) (A)--(C) (C)--(D) (D)--(E);
  \foreach \P in {A,B,C,D,E}
    \filldraw[fill=black!73,draw=black,line width=0.24pt]
      (\P) circle[radius=1.20pt];
  \node[inner sep=0pt] at (0.600000,-0.520000) {(27)};
\end{scope}

\begin{scope}[shift={(2.560000,-7.720000)}]
  \coordinate (A) at (-0.995077,-0.658179);
  \coordinate (B) at (0.069621,0.658179);
  \coordinate (C) at (1.080000,0.016055);
  \coordinate (D) at (-0.069621,-0.317953);
  \coordinate (E) at (-1.080000,0.324171);
  \draw[green!50!black] (A)--(B);
  \draw[blue] (A)--(C) (C)--(E);
  \draw[red] (A)--(D) (A)--(E) (B)--(D);
  \draw[orange] (B)--(C) (B)--(E) (C)--(D) (D)--(E);
  \foreach \P in {A,B,C,D,E}
    \filldraw[fill=black!73,draw=black,line width=0.24pt]
      (\P) circle[radius=1.20pt];
  \node[inner sep=0pt] at (0.000000,-1.120000) {(28)};
\end{scope}

\begin{scope}[shift={(6.080000,-7.720000)}]
  \coordinate (A) at (-1.046087,-0.559775);
  \coordinate (B) at (1.046087,0.120013);
  \coordinate (C) at (-0.446202,-0.839564);
  \coordinate (D) at (0.214486,-0.880000);
  \coordinate (E) at (-0.009379,0.880000);
  \draw[green!50!black] (A)--(B);
  \draw[blue] (A)--(C) (C)--(D);
  \draw[red] (A)--(D) (B)--(D) (B)--(E);
  \draw[orange] (A)--(E) (B)--(C) (C)--(E) (D)--(E);
  \foreach \P in {A,B,C,D,E}
    \filldraw[fill=black!73,draw=black,line width=0.24pt]
      (\P) circle[radius=1.20pt];
  \node[inner sep=0pt] at (0.000000,-1.120000) {(29)};
\end{scope}

\begin{scope}[shift={(9.600000,-7.720000)}]
  \coordinate (A) at (-1.046087,-0.559775);
  \coordinate (B) at (1.046087,0.120013);
  \coordinate (C) at (-0.009379,0.880000);
  \coordinate (D) at (-0.244128,-0.044096);
  \coordinate (E) at (0.214486,-0.880000);
  \draw[green!50!black] (A)--(B);
  \draw[blue] (A)--(C) (C)--(E);
  \draw[red] (A)--(D) (C)--(D) (D)--(E);
  \draw[orange] (A)--(E) (B)--(C) (B)--(D) (B)--(E);
  \foreach \P in {A,B,C,D,E}
    \filldraw[fill=black!73,draw=black,line width=0.24pt]
      (\P) circle[radius=1.20pt];
  \node[inner sep=0pt] at (0.000000,-1.120000) {(30)};
\end{scope}

\begin{scope}[shift={(13.120000,-7.720000)}]
  \coordinate (A) at (-1.080000,-0.185610);
  \coordinate (B) at (1.080000,0.516217);
  \coordinate (C) at (-0.097234,-0.129667);
  \coordinate (D) at (-0.252043,0.346786);
  \coordinate (E) at (0.221439,-0.516217);
  \draw[green!50!black] (A)--(B);
  \draw[blue] (C)--(D) (C)--(E);
  \draw[red] (A)--(C) (A)--(D) (D)--(E);
  \draw[orange] (A)--(E) (B)--(C) (B)--(D) (B)--(E);
  \foreach \P in {A,B,C,D,E}
    \filldraw[fill=black!73,draw=black,line width=0.24pt]
      (\P) circle[radius=1.20pt];
  \node[inner sep=0pt] at (0.000000,-1.120000) {(31)};
\end{scope}

\end{tikzpicture}%
}
\caption{Representatives of the 31 five-point strict distance classes admitting planar crescent realizations. Green, blue, red, and orange edges denote $d_1,d_2,d_3,d_4$, respectively. Each panel is drawn at its own scale.}
\label{fig:allFIVE}
\end{figure}

\begin{theorem}\label{thm:6on6intro}
There are exactly \(6837\) strict distance classes on six vertices. Exactly six of them admit crescent realizations in \(\mathbb R^2\). Representatives of these six classes are shown in Figure~\ref{fig:allSIX}.
\end{theorem}

The proof is given in Section~\ref{sec:six-point}. It follows the five-point argument, with the four-point nonconcyclicity conditions also included in the exclusion step.

\begin{figure}[htbp]
\centering
\resizebox{0.92\linewidth}{!}{%
	\begin{tikzpicture}[
		x=1cm,y=1cm,
		line width=0.34pt,
		line cap=round,
		line join=round,
		font=\small
		]
		
		\begin{scope}[shift={(0.00,0.00)}]
			\coordinate (A) at (1.80000000,0.34157512);
			\coordinate (B) at (1.40469701,1.34511023);
			\coordinate (C) at (-0.72938771,1.34923244);
			\coordinate (D) at (1.79673403,-1.34923244);
			\coordinate (E) at (0.33408472,-0.50100024);
			\coordinate (F) at (-1.80000000,-0.49687804);
			\draw[green!50!black] (A)--(B);
			\draw[blue] (A)--(C) (B)--(D);
			\draw[red] (A)--(D) (A)--(E) (D)--(E);
			\draw[orange] (A)--(F) (B)--(F) (C)--(D) (D)--(F);
			\draw[violet] (B)--(C) (B)--(E) (C)--(E) (C)--(F) (E)--(F);
			\foreach \P in {A,B,C,D,E,F}
			\filldraw[fill=black!73,draw=black,line width=0.24pt]
			(\P) circle[radius=1.35pt];
			\node[inner sep=0pt] at (0,-1.82) {(1)};
		\end{scope}
		
		\begin{scope}[shift={(4.65,0.00)}]
			\coordinate (A) at (-1.61609457,-1.39831581);
			\coordinate (B) at (1.32077571,0.97991484);
			\coordinate (C) at (-1.02200445,0.46748040);
			\coordinate (D) at (0.29731271,0.04739525);
			\coordinate (E) at (1.61609457,-1.40000000);
			\coordinate (F) at (0.00145855,1.40000000);
			\draw[green!50!black] (A)--(B);
			\draw[blue] (A)--(C) (D)--(E);
			\draw[red] (A)--(D) (B)--(C) (B)--(E);
			\draw[orange] (A)--(E) (A)--(F) (C)--(E) (E)--(F);
			\draw[violet] (B)--(D) (B)--(F) (C)--(D) (C)--(F) (D)--(F);
			\foreach \P in {A,B,C,D,E,F}
			\filldraw[fill=black!73,draw=black,line width=0.24pt]
			(\P) circle[radius=1.35pt];
			\node[inner sep=0pt] at (0,-1.82) {(2)};
		\end{scope}
		
		\begin{scope}[shift={(9.30,0.00)}]
			\coordinate (A) at (-1.80000000,-1.39370891);
			\coordinate (B) at (1.80000000,1.39370891);
			\coordinate (C) at (-0.03137946,0.72287379);
			\coordinate (D) at (1.46527052,-0.52772978);
			\coordinate (E) at (-1.25205414,-0.05435082);
			\coordinate (F) at (-0.36610894,-1.19856490);
			\draw[green!50!black] (A)--(B);
			\draw[blue] (A)--(C) (D)--(E);
			\draw[red] (A)--(D) (B)--(E) (B)--(F);
			\draw[orange] (A)--(E) (A)--(F) (C)--(E) (E)--(F);
			\draw[violet] (B)--(C) (B)--(D) (C)--(D) (C)--(F) (D)--(F);
			\foreach \P in {A,B,C,D,E,F}
			\filldraw[fill=black!73,draw=black,line width=0.24pt]
			(\P) circle[radius=1.35pt];
			\node[inner sep=0pt] at (0,-1.82) {(3)};
		\end{scope}
		
		\begin{scope}[shift={(0.00,-3.90)}]
			\coordinate (A) at (1.79665169,-1.40000000);
			\coordinate (B) at (-1.79665169,1.09276325);
			\coordinate (C) at (0.65922160,-1.24777843);
			\coordinate (D) at (-0.11892890,1.40000000);
			\coordinate (E) at (1.58600974,1.35170471);
			\coordinate (F) at (0.69171547,-0.10066782);
			\draw[green!50!black] (A)--(B);
			\draw[blue] (A)--(C) (C)--(F);
			\draw[red] (A)--(D) (B)--(C) (B)--(E);
			\draw[orange] (A)--(E) (B)--(F) (C)--(D) (C)--(E);
			\draw[violet] (A)--(F) (B)--(D) (D)--(E) (D)--(F) (E)--(F);
			\foreach \P in {A,B,C,D,E,F}
			\filldraw[fill=black!73,draw=black,line width=0.24pt]
			(\P) circle[radius=1.35pt];
			\node[inner sep=0pt] at (0,-1.82) {(4)};
		\end{scope}
		
		\begin{scope}[shift={(4.65,-3.90)}]
			\coordinate (A) at (1.79839257,-0.71866836);
			\coordinate (B) at (-1.79839257,1.40000000);
			\coordinate (C) at (-1.59821403,-1.00176356);
			\coordinate (D) at (-0.61160687,-0.69763661);
			\coordinate (E) at (-0.38168496,0.30885372);
			\coordinate (F) at (0.14507837,-1.40000000);
			\draw[green!50!black] (A)--(B);
			\draw[blue] (A)--(C) (B)--(F);
			\draw[red] (C)--(D) (D)--(E) (D)--(F);
			\draw[orange] (A)--(D) (A)--(E) (B)--(C) (B)--(D);
			\draw[violet] (A)--(F) (B)--(E) (C)--(E) (C)--(F) (E)--(F);
			\foreach \P in {A,B,C,D,E,F}
			\filldraw[fill=black!73,draw=black,line width=0.24pt]
			(\P) circle[radius=1.35pt];
			\node[inner sep=0pt] at (0,-1.82) {(5)};
		\end{scope}
		
		\begin{scope}[shift={(9.30,-3.90)}]
			\coordinate (A) at (0.00214816,-1.40000000);
			\coordinate (B) at (0.56563123,-0.98683712);
			\coordinate (C) at (0.77097242,0.34728604);
			\coordinate (D) at (1.61586470,1.40000000);
			\coordinate (E) at (-1.61586470,1.39751951);
			\coordinate (F) at (-0.48708273,-0.14194485);
			\draw[green!50!black] (A)--(B);
			\draw[blue] (A)--(C) (E)--(F);
			\draw[red] (B)--(D) (C)--(E) (D)--(F);
			\draw[orange] (A)--(D) (A)--(E) (B)--(E) (D)--(E);
			\draw[violet] (A)--(F) (B)--(C) (B)--(F) (C)--(D) (C)--(F);
			\foreach \P in {A,B,C,D,E,F}
			\filldraw[fill=black!73,draw=black,line width=0.24pt]
			(\P) circle[radius=1.35pt];
			\node[inner sep=0pt] at (0,-1.82) {(6)};
		\end{scope}
		
	\end{tikzpicture}%
}
\caption{Representatives of all six strict distance classes on six vertices admitting planar crescent realizations. Green, blue, red, orange, and violet edges denote \(d_1,d_2,d_3,d_4,d_5\), respectively. Each panel is drawn at its own scale.}
\label{fig:allSIX}
\end{figure}


The advantage of this algorithmic method is that it can be applied in higher dimensions, however we hope that the running time of $\mathcal{O}(n^{n})$ can still be vastly improved.

In Appendix \ref{app: 5points} we include distance sets and realizable distances for each crescent configuration on four and five points. In Section~\ref{sec:classification}, we introduce some necessary definitions and tools that are used multiple times in our approach. Section~\ref{sec: counting} continues by defining the concept of \textit{strict distance classes} and how to characterize them using distance matrices. This provides the theoretical basis of the enumeration used in Theorems~\ref{thm:3on4}, \ref{thm:27on5}, and~\ref{thm:6on6intro}. We then move to Section~\ref{sec:geomREAL} where we discuss how distance geometry methods may be applied to determine whether a distance set is realizable. Section~\ref{sec:five-point} gives the five-point classification, and Section~\ref{sec:six-point} gives its six-point extension and proves Theorem~\ref{thm:6on6intro}. In Section~\ref{sec:five}, we summarize the computational procedures. Lastly in Section~\ref{sec:future}, we discuss potential future work, including improvements to our algorithm as well as extensions to higher numbers of points and dimensions.

\begin{rem}\label{rem:computational-materials}
The algorithms are described in Appendix~\ref{app: algorithms}. Computational source code is not included with this article. The authors are happy to provide copies of any code referenced in the course of this paper. Please email \href{mailto:Steven.Miller.MC.96@aya.yale.edu,sjm1@williams.edu}{Steven.Miller.MC.96@aya.yale.edu, sjm1@williams.edu}.
\end{rem}

\section{Classification of Crescent Configurations}\label{sec:classification}
In this section we provide the key definitions and theorems that we use to classify crescent configurations. For now, we keep the following definitions general. In section 3 we will discuss the case in which the mentioned distances are measured with the Euclidean metric.

\begin{defi}[Symbolic distance matrix]
Let $\{d_1,\ldots,d_{n-1}\}$ be a set of formal symbols. A \emph{symbolic distance matrix} on \(n\) vertices is a symmetric matrix $M=(m_{ij})_{i,j=1}^n$ such that
\[
m_{ii}=0,
\qquad
m_{ij}\in\{d_1,\ldots,d_{n-1}\}
\quad (i\neq j),
\]
and, for each \(k\in\{1,\ldots,n-1\}\), the symbol \(d_k\) occurs exactly \(k\) times among the entries \(m_{ij}\) with \(i<j\). The symbol \(d_k\) records the distance that is required to occur exactly \(k\) times; at this stage it is not assigned a numerical value.
\end{defi}
A symbolic distance matrix records which pairs of vertices are required to have equal distances, not the numerical distances themselves; see Figure~\ref{fig:symdist}. Equivalently, it is the adjacency matrix of a complete graph whose edges carry formal labels \(d_1,\ldots,d_{n-1}\). Assigning positive, pairwise distinct numerical values to these labels produces a numerical weighted adjacency matrix, but does not by itself guarantee a Euclidean realization in general position. Thus classifying the symbolic patterns and determining which of them admit such realizations are separate tasks.
\begin{figure}
    \centering
    \begin{tikzpicture}[scale=1.5, thick]
    \coordinate (A) at (0,{sqrt(3)});
    \coordinate (B) at (-1,0);
    \coordinate (C) at (1,0);
    \coordinate (D) at (0,-1);

    \draw[red] (A) -- (B) -- (C) -- cycle;

    \draw[blue] (B) -- (D) -- (C);

    \draw[green!60!black] (A) -- (D);

    \node[circle, fill, inner sep=1.5pt, label=above:$A$] at (A) {};
    \node[circle, fill, inner sep=1.5pt, label=left:$B$] at (B) {};
    \node[circle, fill, inner sep=1.5pt, label=right:$C$] at (C) {};
    \node[circle, fill, inner sep=1.5pt, label=below:$D$] at (D) {};
    \end{tikzpicture}
\caption{The same color stands for the same distance, which means $AB=AC=BC,DB=DC$, and $AB,AD,DB$ are distinct from one another. }
    \label{fig:symdist}
\end{figure}


\begin{defi}[Equivalent symbolic distance matrices]\label{def:equivmat}
Let \(M=(m_{ij})\) and \(N=(n_{ij})\) be symbolic distance matrices on \(n\) vertices. We say that \(M\) and \(N\) are \emph{equivalent} if there exists a permutation \(\pi\in S_n\) such that
\[
n_{ij}=m_{\pi(i)\pi(j)}
\]
for every \(i,j\in\{1,\ldots,n\}\).
\end{defi}


\vspace{0.2cm}
\begin{defi}[Distance class]\label{def:disclass}
A \emph{distance class} is an equivalence class of symbolic distance matrices under simultaneous relabeling of the vertices. The distance class represented by \(M\) is denoted by \([M]\).
\end{defi}

\begin{defi}
Let
\[
P=(p_1,\ldots,p_n)
\]
be a crescent configuration determining exactly \(n-1\) distinct distances. For each \(k\in\{1,\ldots,n-1\}\), suppose that \(\lambda_k\) is the distance occurring exactly \(k\) times. Define the symbolic distance matrix \(M(P)\) by
\[
M(P)_{ij}=
\begin{cases}
0,&i=j,\\
d_k,&\lVert p_i-p_j\rVert=\lambda_k.
\end{cases}
\]
The distance class of \(P\) is the equivalence class \([M(P)]\).
\end{defi}

\begin{figure}[h]
\begin{tikzpicture}

\coordinate (C) at (0,0);
\coordinate (P1) at (60:2);   
\coordinate (P2) at (130:2);  
\coordinate (P3) at (170:2);  
\coordinate (P4) at (230:2);  
\coordinate (P5) at (300:2);  
\coordinate (P6) at (10:2);  

\draw [dashed, thin, black] (C) circle (2);

\draw [thick, black] (C) -- (P1);
\draw [thick, black] (C) -- (P2);
\draw [thick, black] (C) -- (P3);
\draw [thick, black] (C) -- (P4);
\draw [thick, black] (C) -- (P5);
\draw [thick, black] (C) -- (P6);


\node [circle, draw, fill=blue!20, minimum size=3pt, inner sep=0pt] at (C) {};

\node [circle, draw, fill=blue!20, minimum size=3pt, inner sep=0pt] at (P1) {};

\node [circle, draw, fill=blue!20, minimum size=3pt, inner sep=0pt] at (P2) {};
\node [circle, draw, fill=blue!20, minimum size=3pt, inner sep=0pt] at (P3) {};
\node [circle, draw, fill=blue!20, minimum size=3pt, inner sep=0pt] at (P4) {};
\node [circle, draw, fill=blue!20, minimum size=3pt, inner sep=0pt] at (P5) {};
\node [circle, draw, fill=blue!20, minimum size=3pt, inner sep=0pt] at (P6) {};


\end{tikzpicture}
\caption{Several points lie at the same distance from a central point. Any four of these outer points are concyclic.}
\end{figure}

\par
\pagebreak
\begin{remark}
For a symbolic distance matrix \(M\), one may associate to each vertex \(i\) the multiset
\[
\rho_i(M)=\{\!\{m_{ij}:j\neq i\}\!\}.
\]
The multiset
\[
\{\!\{\rho_1(M),\ldots,\rho_n(M)\}\!\}
\]
is an invariant of the distance class, but it is not a complete invariant. It may be used as a computational prefilter, but it is not used to define equivalence.

Even though it is tempting to conclude that two crescent configurations are isomorphic when we only consider row profiles,  this is not necessarily the case. One example comes from examining the following two symbolic distance matrices:
\begin{center}
$A = \begin{bmatrix}
0 & 3&4&2&4\\
3&0&4&3&2\\
4&4&0&4&3\\
2&3&4&0&1\\
4&2&3&1&0\\
\end{bmatrix}$ \hspace{0.5in}
$B = \begin{bmatrix}
0&4&4&4&3 \\
4&0&1&3&2 \\
4&1&0&2&3\\
4&3&2&0&4\\
3&2&3&4&0\\
\end{bmatrix}$.
\end{center}

\noindent Here the integer \(k\) denotes the formal label \(d_k\). The row multisets of \(A_1\) and \(A_2\) are, respectively, \(\{\!\{2,3,4,4\}\!\}\) and \(\{\!\{2,3,3,4\}\!\}\), and each occurs only once. In \(B\), these row multisets occur at \(B_4\) and \(B_5\). Consequently, any label-preserving isomorphism would have to send \(A_1\) to \(B_4\) and \(A_2\) to \(B_5\). But \(A_{12}=3\), whereas \(B_{45}=4\). Thus \(A\) and \(B\) are not equivalent.
\par
\end{remark}
\begin{remark}
In fact, our Definitions \ref{def:equivmat} and \ref{def:disclass} imply that, for two matrices $A$ and $B$ in the same distance class, there exists a permutation matrix $P_{\pi}$, such that $P_{\pi}^{\top}AP_{\pi}=B$.  In other words, the configurations given by those equivalent matrices are isomorphic as weighted graphs.

So we only need to consider one representative from each distance class in the following steps. In particular, we can choose $\min_{\pi\in S_n}P_\pi^\top M P_\pi$, here ``$\min$'' refers to the minimum in the lexicographical sense of the upper triangular part of the matrix when expanded row by row.
\end{remark}
That being said, we do now have a tool to encode information of the crescent configuration to some extent. As a result, we will concentrate on determining whether the distance class is realizable, and for each such corresponding adjacency matrix, finding one crescent configuration corresponding to it.

\section{Strict Distance Classes}\label{sec: counting}

As a result of the discussion above, we have a way to classify crescent configurations into distance classes using distance sets. We now further refine these to strict distance classes.
\begin{defi}
For \(S\subseteq\{1,\ldots,n\}\), let \(M[S]\) denote the principal submatrix of \(M\) whose rows and columns are both indexed by \(S\).
\end{defi}
\begin{defi}[Distance-equality pattern]
Let \(Q=(q_{ij})_{i,j=1}^4\) be a symmetric matrix with zero diagonal and formal distance labels off the diagonal. No restriction is imposed on the multiplicities of these labels. We say that \(Q\) has the distance-equality pattern of an isosceles trapezoid if there exist four distinct points \(x_1,x_2,x_3,x_4\), forming a nondegenerate convex isosceles trapezoid in some cyclic order, such that
\[
q_{ij}=q_{k\ell}
\quad\Longleftrightarrow\quad
\lVert x_i-x_j\rVert=\lVert x_k-x_\ell\rVert
\]
for all \(1\leq i<j\leq4\) and \(1\leq k<\ell\leq4\). Here an isosceles trapezoid has a reflection symmetry perpendicular to its bases; rectangles and squares are included.
\end{defi}
\begin{defi}[Strict Distance Class]\label{def:strictclass}
Let \([M]\) be a distance class on \(n\) vertices. We call \([M]\) \emph{strict} if one, equivalently every, representative \(M=(m_{ij})\) satisfies all of the following conditions.
\begin{enumerate}
\item There do not exist a vertex \(u\) and a distance label \(d_k\) such that
    \[
    \#\{v\neq u:m_{uv}=d_k\}\geq4.
    \]

\item There do not exist two distinct vertices \(u,v\) and three distinct vertices
    \[
    w_1,w_2,w_3\notin\{u,v\}
    \]
such that
    \[
    m_{uw_s}=m_{vw_s}
    \qquad (s=1,2,3).
    \]

\item There is no subset
    \[
    S\subseteq\{1,\ldots,n\},
    \qquad |S|=4,
    \]
for which \(M[S]\) has the distance-equality pattern of an isosceles trapezoid.
\end{enumerate}
\end{defi}

The first two forbidden patterns force a failure of planar general position in every realization. In the first case, four points lie on a circle centered at another point. In the second, three points lie on the perpendicular bisector of the same segment. The third restriction is different: it excludes a four-point equality pattern admitting an isosceles-trapezoid realization. The same pattern may admit a different realization in general position, using different numerical distances; see Remark~\ref{rem:parallelogram}. Thus strictness is a combinatorial restriction, not a necessary condition for a crescent realization. All three restrictions are unchanged by simultaneously relabeling the rows and columns. Hence strictness can be checked using one representative of each distance class, for every \(n\).

\subsection{Checking for Strict Distance Classes}

The strictness tests use the complete symbolic matrix. The first two tests are given explicitly in Algorithm~\ref{alg:class}. For the third, the following lemma characterizes the four-point equality patterns to be excluded.

The following lemma provides the necessary and sufficient conditions for a $4 \times 4$ distance matrix to correspond to an isosceles trapezoid.

\begin{lemma}\label{cla:isosceles}
Let \(M\) be a symmetric \(4\times4\) matrix with zero diagonal and formal distance labels off the diagonal. Then \(M\) has the distance-equality pattern of an isosceles trapezoid if and only if, after deleting the diagonal entry from each row, one of the following holds.
\begin{enumerate}
\item There is exactly one row up to permutation of its entries, and one of the following holds.
	\begin{enumerate}
\item The matrix has exactly two distinct labels, with multiplicities \(4\) and \(2\).

\item The matrix has exactly three distinct labels, each with multiplicity \(2\).
	\end{enumerate}

\item There are exactly two distinct rows up to permutation of their entries, each occurring twice, and one of the following holds.
	\begin{enumerate}
\item The matrix has exactly two distinct labels, each with multiplicity \(3\).

\item The matrix has exactly three distinct labels, with multiplicities \(3\), \(2\), and \(1\).

\item The matrix has exactly four distinct labels, with multiplicities \(2\), \(2\), \(1\), and \(1\).
	\end{enumerate}
\end{enumerate}
Here rows are regarded as multisets, all multiplicities are counted among the six entries above the diagonal, and rectangles and squares are included among isosceles trapezoids.
\end{lemma}


\begin{proof}
We first determine the possible distance-equality patterns of an isosceles trapezoid.

Let \(A,B,C,D\) be the vertices of a nondegenerate convex isosceles trapezoid, listed in cyclic order, with \(AB\parallel CD\). After a Euclidean motion, its vertices may be written as
\[
A=(-a,0),\qquad
B=(a,0),\qquad
C=(b,h),\qquad
D=(-b,h),
\]
where \(a,b,h>0\). Put
\[
x=|AB|=2a,\qquad
y=|CD|=2b,
\]
and let
\[
\ell=|AD|=|BC|,
\qquad
q=|AC|=|BD|.
\]
Thus
\[
\ell^2=(a-b)^2+h^2,
\qquad
q^2=(a+b)^2+h^2.
\]
In particular, $q^2-\ell^2=4ab>0$, so $q\neq \ell$.

The four rows of the symbolic distance matrix, after deleting their diagonal entries and regarding them as multisets, have the form
\[
\{\!\{x,\ell,q\}\!\},
\quad
\{\!\{x,\ell,q\}\!\},
\quad
\{\!\{y,\ell,q\}\!\},
\quad
\{\!\{y,\ell,q\}\!\}.
\]

Suppose first that there is only one row up to permutation. By cancellation of the common multiset \(\{\!\{\ell,q\}\!\}\), this is equivalent to \(x=y\). Hence \(a=b\), so the trapezoid is a rectangle. In this case $q^2=x^2+\ell^2$, and consequently \(q\) is distinct from both \(x\) and \(\ell\). If \(x=\ell\), the rectangle is a square, and the two distances have multiplicities \(4\) and \(2\). If \(x\neq\ell\), then \(x,\ell,q\) are pairwise distinct, and each occurs twice. This gives cases \(1(a)\) and \(1(b)\).

Now suppose that there are exactly two rows up to permutation, each occurring twice. Then \(x\neq y\). Since also \(\ell\neq q\), the only possible additional equalities are equalities between one of \(x,y\) and one of \(\ell,q\).

If there are no such equalities, the multiplicities are $2, 2, 1, 1$. If there is exactly one such equality, the multiplicities are $3,2,1$. If there are two such equalities, then, since \(x\neq y\) and \(\ell\neq q\), one must have $\{x,y\}=\{\ell,q\}$, and the multiplicities are \(3\) and \(3\). This proves the necessity of all five alternatives.

We now prove the converse. Let $V=\{1,2,3,4\}$. For each distance label \(\alpha\), let
\[
E_\alpha
=
\bigl\{\{i,j\}\subset V:m_{ij}=\alpha\bigr\},
\]
and let
\[
\deg_\alpha(v)
=
\#\{w\neq v:m_{vw}=\alpha\}.
\]
Thus \(|E_\alpha|\) is the multiplicity of the label \(\alpha\) among the six entries above the diagonal.

Assume first that all four rows are equal as multisets. Then \(\deg_\alpha(v)\) is independent of \(v\) for every label \(\alpha\). If \(|E_\alpha|=r\), the handshaking identity gives $4\deg_\alpha(v)=2r$.

In case \(1(a)\), the label of multiplicity \(4\) has degree \(2\) at every vertex, while the label of multiplicity \(2\) has degree \(1\) at every vertex. Hence the former occupies a \(4\)-cycle and the latter occupies the complementary perfect matching. This is precisely the equality pattern of the four sides and the two diagonals of a square.

In case \(1(b)\), each of the three labels has degree \(1\) at every vertex. Hence each label occupies a perfect matching, and the six edges are partitioned into the three perfect matchings of the complete graph on four vertices. These are precisely the two bases, the two legs, and the two diagonals of a nonsquare rectangle.

Suppose next that there are exactly two row multisets, each occurring twice.

In case \(2(a)\), let the two labels be \(\alpha\) and \(\beta\), each of multiplicity \(3\). The row multiset at a vertex is determined by its \(\alpha\)-degree. Hence the degree sequence of \(E_\alpha\) has the form
\[
r,r,s,s,
\qquad
r\neq s,
\qquad
r+s=3.
\]
The possibility \(\{r,s\}=\{0,3\}\) cannot occur, since a vertex of degree \(3\) would be adjacent to a vertex of degree \(0\). Therefore
\[
\{r,s\}=\{1,2\}.
\]
It follows that \(E_\alpha\) is a path on four vertices. Its complement \(E_\beta\) is also a path on four vertices. After relabeling the vertices, the two label classes are therefore
\[
\{AB,AD,BC\}
\quad\text{and}\quad
\{CD,AC,BD\}.
\]

In case \(2(b)\), let \(\alpha,\beta,\gamma\) have multiplicities \(3,2,1\), respectively. The two endpoints of the unique \(\gamma\)-edge have the same row multiset, since their rows contain \(\gamma\), whereas the other two rows do not. Let \(r\) be the \(\alpha\)-degree at either endpoint of the \(\gamma\)-edge, and let \(s\) be the \(\alpha\)-degree at either of the remaining two vertices. Since \(|E_\alpha|=3\),
\[
2r+2s=6,
\qquad\text{so}\qquad
r+s=3.
\]
Moreover, \(0\leq r\leq2\).

If \(r=0\), then all three \(\alpha\)-edges would have to be supported on the two vertices not incident with the \(\gamma\)-edge, which is impossible. If \(r=2\), then each endpoint of the \(\gamma\)-edge would be joined by an \(\alpha\)-edge to both remaining vertices, producing four \(\alpha\)-edges. Hence \(r=1\) and \(s=2\).

It follows that every vertex has \(\beta\)-degree \(1\), so \(E_\beta\) is a perfect matching. The graph \(E_\alpha\) has degree sequence
\[
1,1,2,2
\]
and is therefore a path on four vertices. The \(\gamma\)-edge joins the two endpoints of this path. After relabeling,
\[
E_\alpha=\{AB,AD,BC\},
\qquad
E_\beta=\{AC,BD\},
\qquad
E_\gamma=\{CD\}.
\]

Finally, consider case \(2(c)\). Let \(\alpha,\beta\) have multiplicity \(2\), and let \(\gamma,\delta\) have multiplicity \(1\). The endpoints of the unique \(\gamma\)-edge form one of the two row types, since precisely their rows contain \(\gamma\). Similarly, the endpoints of the unique \(\delta\)-edge form one row type. These two pairs cannot coincide, since a single edge cannot carry both labels. Hence the \(\gamma\)- and \(\delta\)-edges are disjoint.

Let \(r\) be the \(\alpha\)-degree at the two endpoints of the \(\gamma\)-edge, and let \(s\) be the \(\alpha\)-degree at the two endpoints of the \(\delta\)-edge. Since \(|E_\alpha|=2\),
\[
2r+2s=4,
\qquad\text{so}\qquad
r+s=2.
\]
The possibilities \((r,s)=(0,2)\) and \((2,0)\) are impossible: if, for example, \(r=0\) and \(s=2\), then the two vertices of the second pair would have to be joined by \(\alpha\)-edges to the first pair, forcing positive \(\alpha\)-degree at the first pair. Consequently, $r=s=1$. Thus both \(E_\alpha\) and \(E_\beta\) are perfect matchings, while the two singleton edges form the third perfect matching. After relabeling,
\[
E_\alpha=\{AD,BC\},
\qquad
E_\beta=\{AC,BD\},
\qquad
E_\gamma=\{AB\},
\qquad
E_\delta=\{CD\}.
\]

It remains only to verify that each of these five combinatorial patterns is attained by an actual isosceles trapezoid. For \(a,b,h>0\), consider
\[
A=(-a,0),\qquad
B=(a,0),\qquad
C=(b,h),\qquad
D=(-b,h).
\]
As above, write
\[
x=|AB|=2a,\qquad
y=|CD|=2b,\qquad
\ell=|AD|=|BC|,
\qquad
q=|AC|=|BD|.
\]
Let $\varphi=\frac{1+\sqrt5}{2}$. The following choices give the required exact equality patterns:
\[
\begin{array}{c|c|c}
\text{case}
&
(a,b,h)
&
(x,y,\ell,q)
\\ \hline
1(a)
&
\left(\frac12,\frac12,1\right)
&
(1,1,1,\sqrt2)
\\[1mm]
1(b)
&
\left(\frac12,\frac12,2\right)
&
(1,1,2,\sqrt5)
\\[1mm]
2(a)
&
\left(
\frac12,
\frac{\varphi}{2},
\frac12\sqrt{2+\varphi}
\right)
&
(1,\varphi,1,\varphi)
\\[1mm]
2(b)
&
\left(
\frac12,
1,
\frac{\sqrt3}{2}
\right)
&
(1,2,1,\sqrt3)
\\[1mm]
2(c)
&
\left(
\frac12,
1,
1
\right)
&
\left(
1,2,\frac{\sqrt5}{2},\frac{\sqrt{13}}{2}
\right).
\end{array}
\]
In case \(2(a)\), the identities follow from \(\varphi^2=\varphi+1\). In the final case, the four displayed distances are pairwise distinct. Thus each of the five symbolic patterns is realized by an isosceles trapezoid, with no additional distance equalities.

This proves the converse and completes the proof.
\end{proof}


\begin{remark}
Types M and C were found by Liu, but type R does not appear in his manuscript.  Instead, Liu presents a third type in which the four points form the corners of a parallelogram.  We do not include this here, because this parallelogram type is not a strict distance class; it is in the same distance class as configurations that violate general position. Refer to Remark \ref{rem:parallelogram} for more details on this exclusion. However, together with Liu's result, we obtain a complete classification of distance classes on four points with allowable crescent configurations.
\end{remark}

\begin{figure}[htbp]
\centering
\resizebox{0.85\linewidth}{!}{%
\begin{tikzpicture}[
    x=1cm,y=1cm,
    line width=0.8pt,
    line cap=round,
    line join=round,
    font=\large
]

\begin{scope}
    \coordinate (A) at (-1,{sqrt(7)});
    \coordinate (B) at (1,{sqrt(7)});
    \coordinate (C) at (2,0);
    \coordinate (D) at (-2,0);

    \draw[green!50!black] (A)--(B);
    \draw[blue] (A)--(D) (B)--(C);
    \draw[red] (A)--(C) (B)--(D) (C)--(D);

    \foreach \P in {A,B,C,D}
        \node[circle,fill=black,inner sep=2pt] at (\P) {};

    \node[above left,inner sep=4pt] at (A) {$A$};
    \node[above right,inner sep=4pt] at (B) {$B$};
    \node[below right,inner sep=4pt] at (C) {$C$};
    \node[below left,inner sep=4pt] at (D) {$D$};
\end{scope}

\begin{scope}[xshift=6cm,scale=0.65]
    \coordinate (A) at (0,5);
    \coordinate (B) at (3,-1);
    \coordinate (C) at (3,4);
    \coordinate (D) at (0,0);

    \draw[green!50!black] (A)--(B);
    \draw[blue] (A)--(C) (B)--(D);
    \draw[red] (A)--(D) (B)--(C) (C)--(D);

    \foreach \P in {A,B,C,D}
        \node[circle,fill=black,inner sep=2pt] at (\P) {};

    \node[above left,inner sep=4pt] at (A) {$A$};
    \node[below right,inner sep=4pt] at (B) {$B$};
    \node[above right,inner sep=4pt] at (C) {$C$};
    \node[below left,inner sep=4pt] at (D) {$D$};
\end{scope}

\end{tikzpicture}%
}
\caption{An isosceles trapezoid and a realization of Liu's parallelogram type. Their symbolic distance matrices agree after interchanging $C$ and $D$ in one drawing, but their numerical distance assignments differ. The trapezoid is cyclic, whereas the nonrectangular parallelogram is not. Green, blue, and red denote $d_1,d_2,d_3$, respectively.}
\label{fig:realisom}
\end{figure}

\begin{rem}\label{rem:parallelogram}
To see that Liu's configuration is not in a strict distance class, we can look at the symbolic distance matrix. The symbolic distance matrix for Liu's configuration is
\begin{equation}
\begin{pmatrix}
0 & d_2 & d_1 & d_3 \\
d_2 & 0 & d_3 & d_3 \\
d_1 & d_3 & 0 & d_2 \\
d_3 & d_3 & d_2 & 0
\end{pmatrix}.
\end{equation}
By applying the row (and column) permutation $1 \to 1$, $3\to 2$, $ 4 \to 3$, $2 \to 4$, we get the symbolic distance matrix
\begin{equation}
\begin{pmatrix}
0 & d_1 & d_3 & d_2 \\
d_1 & 0 & d_2 & d_3 \\
d_3 & d_2 & 0 & d_3 \\
d_2 & d_3 & d_3 & 0
\end{pmatrix},
\end{equation}
which has the distance-equality pattern of an isosceles trapezoid. The trapezoid and the parallelogram use different numerical assignments to the same labels.
\end{rem}

\begin{rem}\label{rem:runtime}
The direct enumeration in Algorithm~\ref{alg:class} considers
\[
\frac{\binom n2!}{1!\,2!\cdots(n-1)!}
\]
label assignments before removing repetitions due to vertex relabeling. This number grows rapidly with \(n\). The subsequent exact realizability computations impose a separate computational cost.
\end{rem}

\section{Geometric Realizability of Crescent Configurations} \label{sec:geomREAL}
In the previous sections we developed a way to find potential strict distance classes of crescent configurations on $n$ points. However, it is not clear that there exists a set of points in general position in $\R^d$ such that the distances between the points correspond to each of these distance sets. In this section, we discuss methods of sharpening our previous bounds by taking this consideration into account.

In order to determine whether a distance set can be realized as a crescent configuration in $\R^d$, we must answer two main questions. First, we must determine whether there exist points in $\R^d$ that realize the distance set.  We refer to this property as geometric realizability.  Secondly, we must determine whether such a set of points exists in general position.  These two questions turn out to be closely related, and can be answered using similar techniques.

Fix a representative
\[
M=(m_{ij})
\]
of a distance class on \(n\) vertices. The matrix \(M\) is symbolic: its off-diagonal entries belong to
\[
\{d_1,\ldots,d_{n-1}\},
\]
and the label \(d_k\) occurs exactly \(k\) times among the entries \(m_{ij}\) with \(i<j\). To obtain a numerical distance matrix, we must assign pairwise distinct positive real values to these labels.

We formally define geometric realizability as follows:

\begin{defi}[Numerical distance assignment]
Let \(M=(m_{ij})\) be a symbolic distance matrix on \(n\) vertices. An \emph{admissible distance assignment} is a tuple
\[
\lambda=(\lambda_1,\ldots,\lambda_{n-1})
\]
of pairwise distinct positive real numbers. The associated numerical matrix \(D_M(\lambda)\) is defined by
\[
D_M(\lambda)_{ij}
=
\begin{cases}
0,&i=j,\\
\lambda_k,&m_{ij}=d_k.
\end{cases}
\]
\end{defi}

\begin{remark}
Geometric realizability is invariant under multiplying all entries of an admissible distance assignment by the same positive constant. Consequently, one may normalize one of the distances, for example by setting \(\lambda_1=1\).
\end{remark}

\begin{definition}[Potential numerical distance matrix]
A real matrix \(D=(\delta_{ij})\) is a \emph{potential numerical distance matrix} if
\[
D=D^{\mathsf T},\qquad
\delta_{ii}=0,\qquad
\delta_{ij}>0\quad(i\neq j).
\]
\end{definition}




Many of our explorations involve determining whether a given matrix is a Euclidean distance matrix.  In acknowledgment of this, we define the following class of candidates for Euclidean distance matrices.

\begin{defi}[Euclidean realization]
Let \(M\) be a symbolic distance matrix and let \(\lambda\) be an admissible distance assignment. A \emph{realization} of \((M,\lambda)\) in \(\mathbb R^m\) is an ordered \(n\)-tuple
\[
(p_1,\ldots,p_n)\in(\mathbb R^m)^n
\]
such that
\[
\lVert p_i-p_j\rVert=D_M(\lambda)_{ij}
\]
for all \(i,j\).
\end{defi}

\begin{definition}[Geometric realizability]
A distance class \([M]\) is \emph{geometrically realizable} in \(\mathbb R^m\) if there exists an admissible distance assignment \(\lambda\) and a realization of \((M,\lambda)\) in \(\mathbb R^m\). It admits a \emph{crescent realization} in \(\mathbb R^m\) if such a realization may be chosen in general position.
\end{definition}

\begin{remark}
Burt et al. \cite{SM15} showed that given an integer $n \geq 3$, there exists an $n$-crescent configuration in $\mathbb{R}^{n-2}$.  In this section, we fix $m$ and determine whether a given distance class is geometrically realizable as a crescent configuration in $\R^m$.
\end{remark}

\begin{proposition}
Geometric realizability and the existence of a realization in general position are independent of the representative of a distance class.
\end{proposition}

\begin{proof}
If
\[
N=P_\pi^{\mathsf T}MP_\pi,
\]
then, for every admissible assignment \(\lambda\),
\[
D_N(\lambda)
=
P_\pi^{\mathsf T}D_M(\lambda)P_\pi.
\]
Thus a realization of \(M\) gives a realization of \(N\) by relabeling the points. Relabeling does not affect general position.
\end{proof}

The problem of determining whether a given set of distances can be realized in a space is well-studied, and is known as the distance geometry problem. We can use techniques from distance geometry to determine whether a potential distance matrix has a geometric realization in order to sharpen our previous bound on the number of geometrically realizable configurations on $n$ points in $\R^2$. By the preceding proposition, it is enough to consider one symbolic representative of each distance class. Its numerical distance assignment is still unknown. Geometric realizability asks whether at least one admissible assignment exists; failure of a particular assignment does not exclude the class.

The main tool we use is the Cayley--Menger determinant, which can intuitively be thought of as a way of computing volumes of simplices in Euclidean space.

\begin{defi}
Consider $M$, an $n\times n$ symmetric matrix of the following form:

$$ M=
\left(\begin{matrix}
0 &  d_{1,2}  & \ldots & d_{1,n}\\
d_{2,1}  &  0 & \ldots & d_{2,n} \\
\vdots & \vdots & \ddots &  \vdots\\
d_{n,1} & d_{n,2} & \dots &0
\end{matrix}
\right),
$$
where $d_{i,j} = d_{j,i}$, \(d_{i,j}=d_{j,i}>0\) for \(i\neq j\), with \(d_{i,i}=0\). The \emph{Cayley--Menger determinant} corresponding to $M$, denoted $CM(M)$, is the following determinant:

$$
CM(M) = \det\left(\begin{matrix}
0 &  d_{1,2}^2  & \ldots & d_{1,n}^2 & 1\\
d_{2,1}^2  &  0 & \ldots & d_{2,n}^2 &1\\
\vdots & \vdots & \ddots &  \vdots &\vdots\\
d_{n,1}^2 & d_{n,2}^2 & \dots &0&1\\
1  &   1       &\ldots & 1 & 0
\end{matrix}
\right)
$$

\end{defi}

The Cayley--Menger determinant can be used to determine whether a given $(n+1)\times (n+1)$ matrix is the distance matrix for a set of points in $\R^d$. Failing to adhere to the triangle inequality, for example, could prevent a given matrix from being geometrically realizable as a distance matrix.
\begin{remark}
In the following, we use $[n]$ to refer to the set $\{1,2, \dots, n\}.$
\end{remark}
\begin{defi}
Let $M$ be an $n \times n$ potential distance matrix. Let $S \subset [n]$.  We define $M_S$ to be the minor of $M$ consisting of the rows and columns of $M$ indexed by the elements of $S$.
\end{defi}

\begin{theorem}[Simplex criterion]
\label{thm: cayley n+1}
Let \(D\) be an \((m+1)\times(m+1)\) potential numerical distance matrix. For \(2\leq r\leq m+1\), let
\[
D_r=D[\{1,\ldots,r\}].
\]
Then \(D\) is the distance matrix of \(m+1\) affinely independent points in \(\mathbb R^m\) if and only if
\[
(-1)^r CM(D_r)>0
\]
for every \(r=2,\ldots,m+1\).
\end{theorem}

\begin{proof}
Define the \(m\times m\) matrix \(G=(g_{ij})\) by
\[
g_{ij}
=
\frac12
\left(
D_{1,i+1}^2+D_{1,j+1}^2-D_{i+1,j+1}^2
\right).
\]
For \(1\leq k\leq m\), let \(G_k\) be the leading \(k\times k\) principal submatrix of \(G\). A standard determinant calculation gives
\[
CM(D_{k+1})
=
(-1)^{k+1}2^k\det G_k.
\]
Thus the stated inequalities are equivalent to
\[
\det G_k>0
\qquad(k=1,\ldots,m).
\]
By Sylvester's criterion, this is equivalent to \(G\) being positive definite.

If \(D\) is realized by affinely independent points \(p_1,\ldots,p_{m+1}\), translate so that \(p_1=0\). Then \(G\) is the Gram matrix of
\[
p_2,\ldots,p_{m+1},
\]
and hence is positive definite.

Conversely, if \(G\) is positive definite, write
\[
G=XX^{\mathsf T}.
\]
Taking \(p_1=0\) and taking the rows of \(X\) as \(p_2,\ldots,p_{m+1}\), the defining formula for \(G\) shows that all prescribed distances are realized. Since \(G\) is positive definite, the points are affinely independent.
\end{proof}


Note that if an $(n+1) \times (n+1)$ matrix is a Euclidean distance matrix, then it must be realizable as the Euclidean distance matrix of points in $\R^n$.  Thus, the above result can be thought of as a way of identifying obstacles preventing a set of distances from existing in Euclidean space. When considering potential distance matrices of size greater than $n+1$, we must also take into account dimensional constraints in order to determine whether these matrices can be realized as distance matrices for points in $\R^n$.





The following result gives a complete characterization of whether a given square matrix of size $n >m$ can be expressed nontrivially as the distance matrix of points in $\R^m$.

\begin{theorem}
\label{thm: cayley general}
Let \(D\) be an \(n\times n\) potential numerical distance matrix, where \(n\geq m+1\). Then \(D\) is the distance matrix of \(n\) points whose affine span is \(m\)-dimensional if and only if there exists a set
\[
S_0\subseteq[n],
\qquad |S_0|=m+1,
\]
such that \(D[S_0]\) is the distance matrix of \(m+1\) affinely independent points in \(\mathbb R^m\), and
\[
CM(D[S])=0
\]
for every set \(S\) satisfying
\[
S_0\subseteq S\subseteq[n],
\qquad
|S|\in\{m+2,m+3\}.
\]
\end{theorem}
\begin{proof}
Write \(D=(\delta_{ij})_{i,j=1}^n\). We first record the relation between Cayley--Menger determinants and anchored Gram matrices.

Let
\[
T=\{1,i_1,\ldots,i_k\}\subseteq[n],
\]
where \(i_1,\ldots,i_k\) are distinct elements of \([n]\setminus\{1\}\), and define
\[
G_T=(g_{ab})_{a,b=1}^k,
\qquad
g_{ab}
=
\frac{
\delta_{1i_a}^2+\delta_{1i_b}^2-\delta_{i_ai_b}^2
}{2}.
\]
Then a standard calculation gives
\begin{equation}\label{eq:CM-anchored-Gram}
CM(D[T])
=
(-1)^{k+1}2^k\det G_T.
\end{equation}

Suppose first that \(D\) is realized by points
\[
p_1,\ldots,p_n\in\mathbb R^m
\]
whose affine span is \(m\)-dimensional. We may choose \(m+1\) affinely independent points among them; let their index set be \(S_0\).

If \(S_0\subseteq S\) and
\[
|S|\in\{m+2,m+3\},
\]
then the \(|S|-1\) difference vectors obtained by anchoring the points of \(S\) at any one of its points lie in \(\mathbb R^m\) and are linearly dependent. Their Gram matrix is therefore singular. Equation \eqref{eq:CM-anchored-Gram} gives
\[
CM(D[S])=0.
\]
This proves the necessity.

Conversely, suppose that such a set \(S_0\) exists. After simultaneously relabeling the rows and columns, we may assume that
\[
S_0=\{1,\ldots,m+1\}.
\]
Choose a realization
\[
p_1,\ldots,p_{m+1}\in\mathbb R^m
\]
of \(D[S_0]\). Translating the realization if necessary, assume that \(p_1=0\). Since the points are affinely independent,
\[
p_2,\ldots,p_{m+1}
\]
form a basis of \(\mathbb R^m\). Let
\[
A=\bigl(p_i\cdot p_j\bigr)_{i,j=2}^{m+1}.
\]
Then \(A\) is positive definite and hence invertible. Moreover, by the polarization identity,
\[
A_{ij}
=
\frac{
\delta_{1i}^2+\delta_{1j}^2-\delta_{ij}^2
}{2}.
\]

For each \(r\notin S_0\), define the column vector
\[
b_r
=
\left(
\frac{
\delta_{1i}^2+\delta_{1r}^2-\delta_{ir}^2
}{2}
\right)_{i=2}^{m+1}.
\]
Since \(p_2,\ldots,p_{m+1}\) form a basis, there exists a unique vector \(q_r\in\mathbb R^m\) such that
\[
p_i\cdot q_r=(b_r)_i,
\qquad i=2,\ldots,m+1.
\]

The anchored Gram matrix associated with \(S_0\cup\{r\}\) is
\[
\begin{pmatrix}
A&b_r\\
b_r^{\mathsf T}&\delta_{1r}^2
\end{pmatrix}.
\]
Since
\[
CM\bigl(D[S_0\cup\{r\}]\bigr)=0,
\]
equation \eqref{eq:CM-anchored-Gram} shows that this matrix has determinant zero. The block determinant formula therefore gives
\[
0
=
\det A
\left(
\delta_{1r}^2-b_r^{\mathsf T}A^{-1}b_r
\right).
\]
Since \(\det A\neq0\),
\begin{equation}\label{eq:norm-qr}
\delta_{1r}^2=b_r^{\mathsf T}A^{-1}b_r.
\end{equation}
If \(P\) is the matrix whose columns are \(p_2,\ldots,p_{m+1}\), then
\[
A=P^{\mathsf T}P,
\qquad
P^{\mathsf T}q_r=b_r.
\]
Consequently,
\[
\lVert q_r\rVert^2=b_r^{\mathsf T}A^{-1}b_r
=\delta_{1r}^2.
\]
Furthermore, for \(i=2,\ldots,m+1\),
\[
\begin{aligned}
\lVert p_i-q_r\rVert^2
&=
\lVert p_i\rVert^2+\lVert q_r\rVert^2
-2p_i\cdot q_r\\
&=
\delta_{1i}^2+\delta_{1r}^2
-2(b_r)_i\\
&=
\delta_{ir}^2.
\end{aligned}
\]
Thus every \(q_r\) has the prescribed distances from all points of the reference simplex.

It remains to verify the distances between two points outside \(S_0\). For distinct \(r,s\notin S_0\), put
\[
c_{rs}
=
\frac{
\delta_{1r}^2+\delta_{1s}^2-\delta_{rs}^2
}{2}.
\]
The anchored Gram matrix associated with \(S_0\cup\{r,s\}\) is
\[
\begin{pmatrix}
A&b_r&b_s\\
b_r^{\mathsf T}&\delta_{1r}^2&c_{rs}\\
b_s^{\mathsf T}&c_{rs}&\delta_{1s}^2
\end{pmatrix}.
\]
Its determinant vanishes because
\[
CM\bigl(D[S_0\cup\{r,s\}]\bigr)=0.
\]
Using \eqref{eq:norm-qr} and taking the Schur complement of \(A\), we obtain
\[
\begin{aligned}
0
&=
\det A\,
\det
\begin{pmatrix}
0&
c_{rs}-b_r^{\mathsf T}A^{-1}b_s\\
c_{rs}-b_s^{\mathsf T}A^{-1}b_r&
0
\end{pmatrix}\\
&=
-\det A
\left(
c_{rs}-b_r^{\mathsf T}A^{-1}b_s
\right)^2.
\end{aligned}
\]
Hence
\[
c_{rs}=b_r^{\mathsf T}A^{-1}b_s=q_r\cdot q_s.
\]
It follows that
\[
\begin{aligned}
\lVert q_r-q_s\rVert^2
&=
\lVert q_r\rVert^2+\lVert q_s\rVert^2
-2q_r\cdot q_s\\
&=
\delta_{1r}^2+\delta_{1s}^2-2c_{rs}\\
&=
\delta_{rs}^2.
\end{aligned}
\]

Therefore
\[
p_1,\ldots,p_{m+1},
\quad
\{q_r:r\notin S_0\}
\]
realize all entries of \(D\) in \(\mathbb R^m\). Since \(p_1,\ldots,p_{m+1}\) are affinely independent, the affine span of the full realization is \(m\)-dimensional.
\end{proof}



Thus far, we have focused on techniques that tell us whether a given matrix is geometrically realizable in some given dimension.  However, in order for a potential distance set to correspond to a valid crescent configuration, it must also have a geometric realization that is in general position.  Recall that in order for a crescent configuration to lie in general position in $\R^m$, it must satisfy the following properties:
\begin{enumerate}
\item no $m+1$ points lie on the same hyperplane, and
\item no $m+2$ points lie on the same hypersphere.
\end{enumerate}
As seen above, the Cayley--Menger determinant can be used to determine whether the first condition holds.  A similar determinant can be used to determine whether the second condition holds:

\begin{defi}
Consider $M$, an $n\times n$ symmetric matrix of the following form:

$$ M=
\left(\begin{matrix}
0 &  d_{1,2}  & \ldots & d_{1,n}\\
d_{2,1}  &  0 & \ldots & d_{2,n} \\
\vdots & \vdots & \ddots &  \vdots\\
d_{n,1} & d_{n,2} & \dots &0
\end{matrix}
\right),
$$
where $d_{i,j} = d_{j,i}$, \(d_{i,j}=d_{j,i}>0\) for \(i\neq j\), with \(d_{i,i}=0\). The \emph{Euclidean distance determinant} of $M$, denoted $ED(M)$, is given by the following determinant:

$$
ED(M) = \det\left(
\begin{matrix}
0 &  d_{1,2}^2  & \ldots & d_{1,n}^2 \\
d_{2,1}^2  &  0 & \ldots & d_{2,n}^2\\
\vdots & \vdots & \ddots &  \vdots\\
d_{n,1}^2 & d_{n,2}^2 & \dots &0
\end{matrix}
\right)
.$$
\end{defi}

\begin{theorem}
\label{thm:spheres}
\cite{CS} Let \(D\) be the distance matrix of \(m+2\) points in \(\mathbb R^m\). Then
\[
ED(D)=0
\]
if and only if the points lie on a common hypersphere or in a common hyperplane.
\end{theorem}

\begin{remark}
If every \(m+1\)-point subset is affinely independent, the hyperplane alternative in Theorem~\ref{thm:spheres} is impossible. In that situation, \(ED(D)=0\) is equivalent to cosphericity.
\end{remark}

\noindent The following corollary follows directly from Theorems \ref{thm: cayley general} and \ref{thm:spheres}, and gives a complete characterization as to when a given symmetric matrix can be realized as a distance matrix for points in general position.

\begin{corollary}
\label{cor:general-position-realization}
Let \(D\) be an \(n\times n\) potential numerical distance matrix, where \(n\geq m+2\). Then \(D\) is the distance matrix of \(n\) points in general position in \(\mathbb R^m\) if and only if the following conditions hold.
\begin{enumerate}
\item There exists a set
    \[
    S_0\subseteq[n],
    \qquad |S_0|=m+1,
    \]
such that \(D[S_0]\) is the distance matrix of \(m+1\) affinely independent points in \(\mathbb R^m\).

\item For every \(S\subseteq[n]\) with
    \[
    |S|\in\{m+2,m+3\},
    \]
one has
    \[
    CM(D[S])=0.
    \]

\item For every \(S\subseteq[n]\) with \(|S|=m+1\), one has
    \[
    CM(D[S])\neq0.
    \]

\item For every \(S\subseteq[n]\) with \(|S|=m+2\), one has
    \[
    ED(D[S])\neq0.
    \]
\end{enumerate}
Conditions involving subsets larger than \(n\) are understood to be vacuous.
\end{corollary}

\begin{proof}
Suppose first that \(D\) is realized by points
\[
p_1,\ldots,p_n\in\mathbb R^m
\]
in general position. Any \(m+1\) of these points are affinely independent, so condition (1) holds and, by the Cayley--Menger volume formula, condition (3) holds.

Every set of \(m+2\) or \(m+3\) points in \(\mathbb R^m\) is affinely dependent. Hence the corresponding Cayley--Menger determinants vanish, proving condition (2).

Finally, no \(m+2\) points lie on a common hypersphere or in a common hyperplane. Theorem~\ref{thm:spheres} therefore gives condition (4).

Conversely, conditions (1) and (2) imply, by Theorem~\ref{thm: cayley general}, that \(D\) is the distance matrix of \(n\) points in \(\mathbb R^m\). Condition (3), together with the Cayley--Menger volume formula, implies that no \(m+1\) of these points lie in a common hyperplane. Condition (4) and Theorem~\ref{thm:spheres} imply that no \(m+2\) points lie on a common hypersphere. Thus the realization is in general position.
\end{proof}



\par
In the following section, we apply these criteria to verify exact five-point realizations and combine them with algebraic exclusion arguments to obtain a complete classification of the strict distance classes.


\section{The Five-Point Computation}\label{sec:five-point}

We now describe the two stages of the computation for five points.  The first stage is a finite combinatorial enumeration.  It reduces the set of all symbolic distance matrices to the strict distance classes.  The second stage gives exact algebraic realizations for some of the strict classes.

\subsection{From 118 distance classes to 49 strict distance classes}

We store a symbolic distance matrix by reading its upper-triangular entries in the fixed order
\[
(m_{12},m_{13},m_{14},m_{15},m_{23},m_{24},m_{25},m_{34},m_{35},m_{45}).
\]
The labels $d_1,d_2,d_3,d_4$ must occur with multiplicities $1,2,3,4$, respectively.  Hence the number of symbolic distance matrices with labeled vertices is
\[
\frac{10!}{1!\,2!\,3!\,4!}=12600.
\]
The program generates these matrices by choosing the edge carrying $d_1$, then the two edges carrying $d_2$, and then the three edges carrying $d_3$; the four remaining edges carry $d_4$.

Two matrices define the same distance class when one is obtained from the other by relabeling the five vertices.  To remove duplicates, for each matrix $M$ we apply all $5!$ vertex permutations, always using the same permutation on the rows and columns.  We read the upper-triangular entries of each resulting matrix and retain the lexicographically smallest sequence. This sequence is the canonical representative of the distance class. The exhaustive enumeration gives exactly 118 distance classes.

We next apply the three conditions in Definition~\ref{def:strictclass} to one canonical representative of each class.  In the five-point case, these conditions are checked as follows.
\begin{enumerate}
\item For each vertex, we check whether all four incident edges carry the same distance label.

\item For each pair of vertices $i,j$, we check whether
    \[
    m_{ik}=m_{jk}
    \]
for each of the other three vertices $k$.  If this occurs, the three remaining points would lie on the perpendicular bisector of the segment joining the first two points.

\item For each of the five four-element subsets $S$, we form the $4\times4$ principal submatrix $M[S]$ and apply Lemma~\ref{cla:isosceles}.
\end{enumerate}
These conditions are unchanged by relabeling the vertices, so it is enough to test the canonical representative of each distance class.  The computation leaves exactly 49 strict distance classes.

\subsection{From 49 strict distance classes to 31 realizable classes}

Fix a canonical representative $M=(m_{ij})$ of one of the $49$ strict classes.  Let $\lambda_k$ denote the numerical distance assigned to the label $d_k$, and put $s_k=\lambda_k^2$.  Since realizability is unchanged under a common positive scaling of all distances, we normalize $s_1=1$.  The unknown numerical data are therefore $s_2,s_3,s_4$.

For a subset $S$ of the vertices, the Cayley--Menger determinant of $D_M(\lambda)[S]$ is a polynomial in $s_2,s_3,s_4$: whenever the symbolic entry is $d_k$, the corresponding squared entry in the determinant is $s_k$.  For each of the classes realized below, we use the reference triangle $S_0=\{1,2,3\}$.  The realizability equations are
\[
CM\bigl(D_M(\lambda)[\{1,2,3,4\}]\bigr)=0,
\qquad
CM\bigl(D_M(\lambda)[\{1,2,3,5\}]\bigr)=0,
\]
and
\[
CM\bigl(D_M(\lambda)\bigr)=0.
\]
Together with the condition that $D_M(\lambda)[S_0]$ is the distance matrix of a nondegenerate triangle, Theorem~\ref{thm: cayley general} shows that these equations are sufficient for a realization in $\mathbb R^2$.

A preliminary numerical search was used only to locate possible solutions. Each retained solution was then replaced by an exact algebraic certificate. Such a certificate consists of an irreducible polynomial $g\in\mathbb Q[t]$, a specified real root $\alpha$ of $g$, and polynomials $r_2,r_3,r_4\in\mathbb Q[t]$ such that
\[
s_2=r_2(\alpha),
\qquad
s_3=r_3(\alpha),
\qquad
s_4=r_4(\alpha).
\]
For one of the classes, all three squared distances may be chosen rationally. For the remaining classes in our list, the degree of $g$ is at most six. The complete exact certificates are available from the authors upon request; see Remark~\ref{rem:computational-materials}.

The verification is exact.  Since \(g\) is irreducible and \(g(\alpha)=0\), a polynomial \(P\in\mathbb Q[t]\) satisfies \(P(\alpha)=0\) if and only if its remainder upon division by \(g\) is zero. For each certificate, the program checks the following conditions.
\begin{enumerate}
\item The numbers $1,s_2,s_3,s_4$ are positive and pairwise distinct.

\item The reference triangle is nondegenerate, and the two four-point Cayley--Menger determinants and the five-point Cayley--Menger determinant displayed above vanish exactly.

\item For every three-point subset $S$,
    \[
    CM\bigl(D_M(\lambda)[S]\bigr)\neq0.
    \]
Thus no three points are collinear.

\item For every four-point subset $S$,
    \[
    ED\bigl(D_M(\lambda)[S]\bigr)\neq0.
    \]
Thus no four points are concyclic.
\end{enumerate}
The first two checks and Theorem~\ref{thm: cayley general} give an exact planar realization, while the last two checks give general position. The program verifies such a certificate for $31$ of the $49$ strict distance classes.

\begin{proposition}\label{prop:31on5}
There are at least $31$ pairwise not-equivalent strict distance classes that admit crescent realizations on five points in $\mathbb R^2$.
\end{proposition}

\begin{proof}
	Table~\ref{tab: matrixTABLE} lists \(31\) distinct canonical representatives of strict distance classes. The corresponding exact squared distances are given in Table~\ref{tab: distanceTABLE}, with algebraic parameters specified in Table~\ref{tab:algebraic-parameters}. Each assignment passes the four checks above, implemented in Algorithm~\ref{alg: check4}, and therefore gives a planar crescent realization. Since the canonical representatives are distinct, these realizations belong to pairwise inequivalent distance classes.
\end{proof}
The construction stage leaves \(18\) strict distance classes without a realization certificate. We exclude these classes in the following subsection.

\subsection{Exact exclusion of the remaining classes}

We briefly recall the small amount of Gr\"obner-basis theory used in the computation.

Let
\[
R=\mathbb Q[x_1,\ldots,x_N].
\]
For polynomials \(f_1,\ldots,f_r\in R\), the ideal generated by them is
\[
\langle f_1,\ldots,f_r\rangle
=
\left\{
\sum_{i=1}^r a_i f_i:a_i\in R
\right\}.
\]
Thus for every polynomial \( g \in \langle f_1,\ldots,f_r\rangle\), the equation \(g=0\) follows from the equation \(f_1=\cdots=f_r=0\).

Fix a monomial order, that is, a total well-order of the monomials, with \(1\) smallest, that is compatible with multiplication. For a nonzero polynomial \(f\), let \(\operatorname{LM}(f)\) denote its largest monomial.

\begin{defi}[Gr\"obner basis]
Let \(I\subseteq R\) be an ideal. A finite generating set
\[
G=\{g_1,\ldots,g_t\}
\]
of \(I\) is a \emph{Gr\"obner basis} if, for every nonzero \(f\in I\), the leading monomial \(\operatorname{LM}(f)\) is divisible by \(\operatorname{LM}(g_j)\) for some \(j\).
\end{defi}

For every fixed monomial order, Buchberger's algorithm computes a Gr\"obner basis of a finitely generated polynomial ideal; see \cite[Chapter~2, Sections~5--7]{CLO15}. We use the algorithm only through the following elementary consequence.

\begin{lemma}\label{lem:gb-nonvanishing}
Let
\[
f_1,\ldots,f_r,h\in
\mathbb Q[x_1,\ldots,x_N],
\]
and introduce one additional variable \(u\). Set
\[
J=
\langle
f_1,\ldots,f_r,uh-1
\rangle
\subseteq
\mathbb Q[u,x_1,\ldots,x_N].
\]
If a Gr\"obner basis of \(J\) is \(\{1\}\), then there is no \(z\in\mathbb C^N\) such that
\[
f_1(z)=\cdots=f_r(z)=0
\qquad\text{and}\qquad
h(z)\neq0.
\]
\end{lemma}

\begin{proof}
Since \(1\in J\), there exist polynomials \(A_1,\ldots,A_r,B\) such that
\[
1=
\sum_{i=1}^r A_i f_i+B(uh-1).
\]
If \(z\) satisfied the stated conditions, we could set
\[
u=h(z)^{-1}.
\]
Evaluating the identity at \((u,z)\) would give \(1=0\), a contradiction.
\end{proof}
\begin{remark}
For example, the equations \(x=0\) and \(x-1=0\) are inconsistent because
\[
1=x-(x-1)
\]
belongs to the ideal generated by \(x\) and \(x-1\). The output \(\{1\}\) in the multivariable computation is the analogous certificate.
\end{remark}
Fix one of the remaining strict distance classes, with canonical representative \(M\). Write
\[
s_k=\lambda_k^2,
\qquad k=1,2,3,4,
\]
and normalize \(s_1=1\).

For the fixed representative \(M\), define the polynomial matrix
\[
Q(s)_{ij}=
\begin{cases}
0,&i=j,\\
s_k,&m_{ij}=d_k,
\end{cases}
\qquad s_1=1.
\]
For a vertex subset \(T\), put
\[
C_T(s)=
\det
\begin{pmatrix}
Q(s)[T]&\mathbf 1\\
\mathbf 1^{\mathsf T}&0
\end{pmatrix}.
\]
The entries of \(Q(s)\) are squared distances and are not squared again. At an admissible assignment \(s_k=\lambda_k^2\), this polynomial equals \(CM(D_M(\lambda)[T])\).

Set
\[
F_{M,4}=C_{\{1,2,3,4\}}(s),\qquad
F_{M,5}=C_{\{1,2,3,5\}}(s),\qquad
F_{M,45}=C_{\{1,2,3,4,5\}}(s).
\]
These belong to \(\mathbb Q[s_2,s_3,s_4]\). Define
\[
H_M=
s_2s_3s_4
\prod_{1\leq i<j\leq4}(s_i-s_j)
\prod_{\substack{T\subseteq[5]\\|T|=3}}C_T(s),
\qquad s_1=1.
\]
Introduce an additional variable \(u\), and let
\[
J_M
=
\left\langle
F_{M,4},
F_{M,5},
F_{M,45},
uH_M-1
\right\rangle
\subseteq
\mathbb Q[u,s_2,s_3,s_4].
\]
\begin{proposition}\label{prop:gb-exclusion}
If a Gr\"obner basis of \(J_M\) is \(\{1\}\), then the distance class represented by \(M\) admits no crescent realization in \(\mathbb R^2\).
\end{proposition}

\begin{proof}
Suppose that such a realization existed. Its squared distances would give positive, pairwise distinct real numbers
\[
s_1=1,s_2,s_3,s_4.
\]
Since the realization lies in \(\mathbb R^2\), its two four-point Cayley--Menger determinants and its five-point Cayley--Menger determinant would vanish. Thus
\[
F_{M,4}=F_{M,5}=F_{M,45}=0.
\]
Since the points are in general position, no three are collinear, and hence every three-point Cayley--Menger determinant occurring in \(H_M\) is nonzero. Therefore \(H_M\neq0\). Setting \(u=H_M^{-1}\) would give a common zero of the generators of \(J_M\), contradicting Lemma~\ref{lem:gb-nonvanishing}.
\end{proof}
The excluded classes are
\[
3,4,6,8,9,11,12,15,17,19,22,30,33,37,40,41,42,44.
\]
\begin{theorem}\label{thm:exactly31}
Exactly \(31\) of the \(49\) five-point strict distance classes admit crescent realizations in \(\mathbb R^2\).
\end{theorem}

\begin{proof}
Exact algebraic realizations were given for \(31\) distinct strict distance classes. For each of the remaining \(18\) classes, the exact Gr\"obner-basis computation gives \(\{1\}\). Proposition \ref{prop:gb-exclusion} excludes a crescent realization of each of these classes.
\end{proof}

\section{The Six-Point Computation}\label{sec:six-point}

Algorithm~\ref{alg:class}, applied to six vertices, gives \(53528\) distance classes, of which \(6837\) are strict. The realization checks in Algorithm~\ref{alg: check4} already apply to six points without any change. The exclusion step, however, requires a condition that was not needed to finish the five-point computation.

For a six-point representative \(M\), retain the notation \(Q(s)\) and \(C_T(s)\) from Section~\ref{sec:five-point}, now with five squared distances \(s_1=1,s_2,s_3,s_4,s_5\), and put
\[
E_T(s)=\det Q(s)[T].
\]
Subscripts such as \(1234\) will denote vertex subsets. With \(\{1,2,3\}\) as the reference triangle, the planar realizability equations are
\begin{equation}\label{eq:six-reference}
	C_{1234}=C_{1235}=C_{1236}=C_{12345}=C_{12346}=C_{12356}=0.
\end{equation}
The direct extension of the five-point nonvanishing product is
\[
H_M^{(6)}=
\left(\prod_{k=2}^{5}s_k\right)
\left(\prod_{1\leq i<j\leq5}(s_i-s_j)\right)
\left(\prod_{\substack{T\subseteq[6]\\|T|=3}}C_T(s)\right).
\]
Computing the unit-ideal test for the six polynomials in \eqref{eq:six-reference}, together with \(uH_M^{(6)}-1\), excludes \(6830\) strict classes and leaves seven classes undecided by that test. Six of these have exact assignments accepted by Algorithm~\ref{alg: check4}. The remaining class is represented by
\begin{equation}\label{eq:six-cyclic-matrix}
	M_0=
	\begin{pmatrix}
		0&d_1&d_3&d_4&d_4&d_5\\
		d_1&0&d_3&d_5&d_5&d_5\\
		d_3&d_3&0&d_2&d_4&d_2\\
		d_4&d_5&d_2&0&d_4&d_5\\
		d_4&d_5&d_4&d_4&0&d_3\\
		d_5&d_5&d_2&d_5&d_3&0
	\end{pmatrix}.
\end{equation}
For example, let \(t\) be the root of \(3t^3-9t^2+6t-1=0\) in \((1/4,1/3)\), and set
\[
(s_1,s_2,s_3,s_4,s_5)=
\bigl(1,6t^2-5t+1,3t^2-3t+1,3t^2-5t+2,t\bigr).
\]
This assignment has positive, pairwise distinct squared distances, satisfies \eqref{eq:six-reference}, and has \(C_T<0\) for every triple. It therefore gives a genuine noncollinear planar realization. Nevertheless, the distance equalities of \(M_0\) force vertices \(1,2,3,5\) to be concyclic. Thus this class is not accepted by the realization verifier; it is only left undecided by the weaker exclusion test.

The forced circle also has a short geometric explanation. Suppose there were a crescent realization \(p_1,\ldots,p_6\) of \(M_0\), and normalize
\[
p_1=(-1/2,0),\qquad p_2=(1/2,0).
\]
The equalities \(|p_1p_4|=|p_1p_5|\) and \(|p_2p_4|=|p_2p_5|\) imply that the distinct points \(p_4,p_5\) are reflections across the horizontal axis. The equalities \(|p_1p_3|=|p_2p_3|\) and \(|p_1p_6|=|p_2p_6|\) also give
\[
p_4=(x,y),\qquad p_5=(x,-y),\qquad p_3=(0,a),\qquad p_6=(0,b).
\]
Using, respectively, \(|p_2p_4|=|p_2p_6|\), \(|p_3p_4|=|p_3p_6|\), and \(|p_3p_5|=|p_1p_4|\), we obtain
\[
x^2+y^2-x=b^2,\qquad
x^2+y^2-2ay+2ab=b^2,\qquad
a^2+2ay=x+\frac14.
\]
Subtracting the first two equations and substituting into the third gives \(a^2+2ab=1/4\). Since \(|p_1p_6|=|p_2p_6|=|p_4p_6|\), reflection shows that \(p_1,p_2,p_5\) lie on the circle centered at \((0,-b)\) with squared radius \(1/4+b^2\). But
\[
|p_3-(0,-b)|^2=(a+b)^2=\frac14+b^2,
\]
so \(p_3\) lies on the same circle, contradicting general position.

Rather than add a test specific to \(M_0\), we incorporate all four-point nonconcyclicity conditions by replacing \(H_M^{(6)}\) with
\begin{equation}\label{eq:six-full-product}
	\widetilde H_M^{(6)}=
	H_M^{(6)}\prod_{\substack{T\subseteq[6]\\|T|=4}}E_T(s).
\end{equation}
Every crescent realization satisfies \(\widetilde H_M^{(6)}\neq0\), by the same determinant criterion used in Algorithm~\ref{alg: check4}. Lemma~\ref{lem:gb-nonvanishing} therefore applies unchanged to the ideal generated by \eqref{eq:six-reference} and \(u\widetilde H_M^{(6)}-1\). The previously excluded classes remain excluded, and the strengthened computation gives the unit ideal for \(M_0\) as well. Thus the six-point procedure uses the same exclusion argument, with the reference equations indexed over six vertices and the additional factors in \eqref{eq:six-full-product}.

\begin{proof}[Proof of Theorem~\ref{thm:6on6intro}]
	The six classes illustrated in Figure~\ref{fig:allSIX} have distinct canonical representatives and exact assignments passing Algorithm~\ref{alg: check4}. The strengthened exclusion test excludes each of the other \(6831\) representatives. Together with the enumeration at the beginning of this section, this proves Theorem~\ref{thm:6on6intro}.
\end{proof}

The six representatives, their exact algebraic distance assignments, and the exclusion and verification outputs are available from the authors upon request. The exceptional class explains the difference from the five-point computation: the earlier exclusion test used noncollinearity but not nonconcyclicity, and was sufficient for the eighteen five-point classes that required exclusion. Its success on those inputs did not make it a complete realizability test. The strictness conditions likewise do not replace the numerical general-position checks. A uniform treatment for larger inputs is described in Remark~\ref{rem:uniform-decision}.

\section{Finding Crescent Configurations}\label{sec:five}

Sections~\ref{sec:classification}--\ref{sec:geomREAL} provide the definitions and realizability criteria used in our computations. Sections~\ref{sec:five-point} and~\ref{sec:six-point} give the five- and six-point classifications. Here we summarize the computational steps; explicit procedures are given in Appendix~\ref{app: algorithms}.

The computation has three tasks: list the distance classes, retain the strict classes, and determine which retained classes admit crescent realizations. The first two tasks involve only distance labels. Numerical distances enter only in the third task.

We begin with symmetric symbolic matrices on \(n\) vertices. The label \(d_k\) occupies exactly \(k\) positions above the diagonal, for \(k=1,\ldots,n-1\). Thus the number of matrices before identifying different vertex numberings is
\[
\frac{\binom{n}{2}!}{1!\,2!\cdots(n-1)!}.
\]
This gives \(60\) matrices for \(n=4\) and \(12600\) matrices for \(n=5\).

We next remove repetitions caused by different numberings of the vertices. For each matrix, we consider all vertex permutations, applying the same permutation to its rows and columns. We read the entries above the diagonal in a fixed order and retain the smallest resulting label sequence. Here sequences are compared from left to right: at the first position where they differ, the sequence with the smaller label index comes first.

Two matrices are placed in the same distance class exactly when their retained sequences agree. The distance labels themselves are not permuted, and entries within individual rows are not sorted independently.

\begin{remark}\label{rem:counting}
One cannot obtain the number of distance classes simply by dividing the number of matrices by \(n!\). Some matrices are unchanged by a nonidentity permutation of the vertices, so the number of distinct numberings need not be the same for every class.
\end{remark}

We then test one representative of each class against the three conditions in Definition~\ref{def:strictclass}. These tests concern the complete symbolic matrix: repeated labels incident to one vertex, equal labels at corresponding positions for a common base, and the four-point patterns in Lemma~\ref{cla:isosceles}. All three tests are unchanged by relabeling the vertices.

Algorithm~\ref{alg:class} gives the following exact counts:
\[
\begin{array}{c|r|r|r}
n&
\text{Matrices with numbered vertices}&
\text{Distance classes}&
\text{Strict distance classes}\\
\hline
4&60&4&3\\
5&12600&118&49
\end{array}
\]
Strictness is the combinatorial restriction defined in Section~\ref{sec: counting}; it is not a replacement for the general-position checks on a numerical realization.

For each retained representative \(M\), we next seek numerical distances realizing its labels. We use squared distances \(s_k=\lambda_k^2\) and normalize \(s_1=1\). A numerical search may suggest suitable values, but only exact assignments passing the checks in Algorithm~\ref{alg: check4} are accepted.

The verification has a concrete meaning. First, the squared distances must be positive and pairwise distinct. Next, the reference triangle must be nondegenerate, and the Cayley--Menger equations from Theorem~\ref{thm: cayley general} must hold. Finally, the three-point and four-point determinant tests exclude collinearity and concyclicity. Passing these checks proves that the class admits a planar crescent realization.

A failed assignment does not exclude its distance class: another choice of distances might work. Nonexistence is therefore checked separately. For the five-point classes, we use Algorithm~\ref{alg:exclude5}, which tests the polynomial contradiction described in Section~\ref{sec:five-point}. This contradiction applies to every possible admissible assignment, not merely to the values tried during a numerical search.

All three four-point strict classes admit realizations. For five points, exact realization certificates are available for \(31\) of the \(49\) strict classes, and exact exclusion certificates are available for the other \(18\). Consequently, the five-point result is an exact classification, as stated in Theorem~\ref{thm:exactly31}.

Appendix~\ref{app: 5points} lists the representatives of the \(31\) realizable five-point classes and displays one numerical assignment for each. The exact algebraic assignments used in the verification are available from the authors upon request.

\begin{remark}
We count distance classes, not individual coordinate realizations. A class may admit several, or even infinitely many, numerical realizations and is still counted only once. One verified realization proves existence for that class; an exclusion certificate rules out all of its realizations.
\end{remark}

\section{Future Work}\label{sec:future}

\subsection{Improving the Algorithm}

Our classification concerns strict distance classes rather than all distance classes. The first two restrictions in Definition~\ref{def:strictclass} force a failure of planar general position. The third is a more restrictive combinatorial filter and may discard a class that has a valid realization.

A natural next step is to examine the classes discarded only by the third restriction, using the exact realizability tests without that preliminary filter. This would extend the classification beyond the strict classes. No further subdivision by vertex relabeling is needed: that equivalence is already built into our definition of a distance class.

\subsection{Further Explorations in the Plane}\label{sub:further}

For larger numbers of points, both the enumeration of distance classes and the exact realizability computations become more demanding. One possible improvement is to identify equivalent partial label assignments before completing the matrices, rather than generating every numbered matrix and removing repetitions afterwards.

It is important to distinguish this computational problem from the original existence question. Excluding all strict classes for a given number of points would not exclude crescent configurations in non-strict classes.

For each \(n\), let \(c_n\) denote the number of distance classes admitting planar crescent realizations, without the strictness restriction. This is a finite number because there are only finitely many symbolic distance matrices on \(n\) vertices. The conjecture recalled in the introduction is equivalent to \(c_n=0\) for all sufficiently large \(n\). Our strict classification determines part of this broader count.

\subsection{Extensions to Higher Dimensions}\label{sub:high}

The realizability criteria in Section~\ref{sec:geomREAL} apply in \(\mathbb R^m\), not only in the plane. The definition of a distance class also remains unchanged: equivalent symbolic matrices differ only by a relabeling of the vertices.

The preliminary strictness conditions, however, were chosen for the plane and should not be imposed unchanged in higher dimensions. A higher-dimensional computation should instead use restrictions appropriate to the requirements that no \(m+1\) points lie in a hyperplane and no \(m+2\) points lie on a hypersphere.

One question is which distance classes without planar crescent realizations admit such realizations in \(\mathbb R^3\). More generally, one may compare the realizable distance classes in different dimensions and determine the smallest dimension admitting a general-position realization of a fixed class.

\begin{remark}[A uniform real decision problem]\label{rem:uniform-decision}
	Adding nonconcyclicity factors removes the obstruction in Section~\ref{sec:six-point}, but a nonunit Gr\"obner basis still does not establish realizability: polynomial equations can have complex solutions without any real solution, as \(x^2+1=0\) illustrates. Positivity must also be imposed, rather than replaced by nonvanishing. These issues can be handled uniformly for larger numbers of points and for higher ambient dimensions.
	
	Fix \(m\geq2\) and \(n\geq m+1\), and let \(M\) be a distance-class representative on \(n\) vertices. Use the same squared-distance matrix \(Q(s)\), with \(s_1=1\), and the determinants \(C_T(s)\) and \(E_T(s)\) defined above. Set \(S_0=\{1,\ldots,m+1\}\) and form the reference Gram matrix
	\[
	A(s)=\left(\frac{Q(s)_{1i}+Q(s)_{1j}-Q(s)_{ij}}2\right)_{i,j=2}^{m+1}.
	\]
	For \(1\leq r\leq m\), let \(A_r(s)\) be its upper-left \(r\times r\) submatrix. Positive definiteness of \(A\), which makes \(S_0\) a nondegenerate Euclidean simplex, is equivalent to the finitely many inequalities \(\det A_r>0\). Consider the following conditions on real variables \(s_2,\ldots,s_{n-1}\):
	\[
	\begin{aligned}
		&s_k>0 &&(1\leq k\leq n-1),\\
		&s_i\neq s_j &&(1\leq i<j\leq n-1),\\
		&\det A_r(s)>0 &&(1\leq r\leq m),\\
		&C_T(s)=0 &&(S_0\subseteq T\subseteq[n],\ |T|\in\{m+2,m+3\}),\\
		&C_T(s)\neq0 &&(T\subseteq[n],\ |T|=m+1),\\
		&E_T(s)\neq0 &&(T\subseteq[n],\ |T|=m+2).
	\end{aligned}
	\]
	These conditions are necessary and sufficient for a general-position realization of \(M\) in \(\mathbb R^m\). Indeed, Theorem~\ref{thm: cayley general} gives the realization from the reference simplex and the vanishing equations. The last two conditions then exclude, respectively, \(m+1\) points in a hyperplane and \(m+2\) points on a hypersphere. Conversely, every admissible realization satisfies all the displayed conditions. Fixing \(S_0\) loses no admissible realization, since every \(m+1\)-point subset is affinely independent in general position.
	
	A set specified by finitely many polynomial equalities and inequalities is called \emph{semialgebraic}. Real quantifier elimination decides whether such a system has a real solution; here the quantifier is simply ``there exist \(s_2,\ldots,s_{n-1}\).'' For example, cylindrical algebraic decomposition partitions the real variable space into finitely many pieces on each of which every input polynomial has constant sign, and supplies exact sample points. Checking these samples determines whether any piece satisfies the whole system; see \cite{BPR06}. This works both for isolated solutions and for continuous families, so it does not require that the possible distance assignments form a finite set.
	
	Thus the Gr\"obner-basis test can be retained as an initial exclusion step, followed, when necessary, by a complete real decision procedure for the displayed system. Increasing \(n\) or \(m\) changes the sizes and index sets, not the types of geometric conditions. This gives completeness in principle, not a guarantee of practical running time; an interrupted computation must remain undecided.
\end{remark}

\clearpage

\appendix

\section{Computational Procedures}\label{app: algorithms}

The computation is divided into three procedures. Algorithm~\ref{alg:class} lists the strict distance classes. Algorithm~\ref{alg: check4} checks one proposed exact distance assignment. Algorithm~\ref{alg:exclude5} attempts to exclude each remaining five-point class. Rejecting an assignment in the second procedure does not exclude its distance class.

\subsection{Enumeration and the strictness tests}

A symbolic matrix is stored by its upper-triangular labels, read row by row. For five vertices, the order is
\[
(m_{12},m_{13},m_{14},m_{15},m_{23},
 m_{24},m_{25},m_{34},m_{35},m_{45}).
\]
Sequences are compared from left to right, using the order \(d_1<d_2<\cdots<d_{n-1}\) on the labels. At the first position where two sequences differ, the one with the smaller label index comes first. This does not impose any ordering on the numerical distances.

Relabeling vertices always means applying the same permutation to both the rows and the columns. We do not independently sort the entries within each row.

\begin{algorithm}[htbp]
\caption{Enumeration of strict distance classes}
\label{alg:class}
\small
\begin{algorithmic}
\Require An integer \(n\geq4\).
\Ensure A list \(\mathcal C\) of distance-class representatives and its sublist \(\mathcal S\) of strict representatives.

\State \(\mathcal C\gets\) empty list.

\ForAll{placements of \(d_1,\ldots,d_{n-1}\) on the \(\binom n2\) pairs, with \(d_k\) used exactly \(k\) times}
\State Form the corresponding symmetric symbolic matrix \(M\).
\State \(B\gets M\).
\ForAll{permutations \(\pi\) of \(\{1,\ldots,n\}\)}
\State \(N\gets(m_{\pi(i)\pi(j)})_{i,j=1}^n\).
\If{the upper-triangular sequence of \(N\) precedes that of \(B\)}
\State \(B\gets N\).
\EndIf
\EndFor
\If{\(B\) is not already in \(\mathcal C\)}
\State Append \(B\) to \(\mathcal C\).
\EndIf
\EndFor

\State Sort \(\mathcal C\) by the upper-triangular sequences.
\State \(\mathcal S\gets\) empty list.

\ForAll{\(M=(m_{ij})\) in \(\mathcal C\)}
\State \(\mathrm{discard}\gets\mathrm{false}\).

\ForAll{vertices \(i\)}
\If{some label occurs at least four times in row \(i\), excluding its diagonal}
\State \(\mathrm{discard}\gets\mathrm{true}\).
\EndIf
\EndFor

\ForAll{pairs \(i<j\) in \(\{1,\ldots,n\}\)}
\State \(c\gets\#\{k\in[n]\setminus\{i,j\}:m_{ik}=m_{jk}\}\).
\If{\(c\geq3\)}
\State \(\mathrm{discard}\gets\mathrm{true}\).
\EndIf
\EndFor

\ForAll{four-element subsets \(T\subseteq\{1,\ldots,n\}\)}
\If{\(M[T]\) satisfies one of the cases in Lemma~\ref{cla:isosceles}}
\State \(\mathrm{discard}\gets\mathrm{true}\).
\EndIf
\EndFor

\If{\(\mathrm{discard}=\mathrm{false}\)}
\State Append \(M\) to \(\mathcal S\).
\EndIf
\EndFor

\State Number the entries of \(\mathcal S\) starting from \(1\).
\State \Return \((\mathcal C,\mathcal S)\).
\end{algorithmic}
\end{algorithm}

The placements in the first loop can be generated by choosing one pair for \(d_1\), then two unused pairs for \(d_2\), and so on; the final \(n-1\) pairs receive \(d_{n-1}\). Thus every assignment with the required multiplicities is considered.

The four-point computation gives
\[
|\mathcal C|=4,\qquad |\mathcal S|=3,
\]
and the five-point computation gives
\[
|\mathcal C|=118,\qquad |\mathcal S|=49.
\]
All five-point class numbers below refer to this ordered list of \(49\) strict representatives.

\subsection{Exact verification of a proposed realization}

Fix a representative \(M\). Write \(s_k=\lambda_k^2\) and normalize \(s_1=1\). Define the squared-distance matrix
\[
Q(s)_{ij}=
\begin{cases}
0,&i=j,\\
s_k,&m_{ij}=d_k.
\end{cases}
\]
For a vertex subset \(T\), put
\[
C_T(s)=
\det
\begin{pmatrix}
Q(s)[T]&\mathbf 1\\
\mathbf 1^{\mathsf T}&0
\end{pmatrix},
\qquad
E_T(s)=\det Q(s)[T].
\]
These are the Cayley--Menger and Euclidean distance determinants of Section~\ref{sec:geomREAL}. The entries of \(Q(s)\) are already squared distances and are not squared again.

\begin{algorithm}[htbp]
\caption{Checking one exact distance assignment}
\label{alg: check4}
\small
\begin{algorithmic}
\Require A strict representative \(M\) on \(n\geq4\) vertices and an assignment \(s=(1,s_2,\ldots,s_{n-1})\) of exactly specified real algebraic numbers.
\Ensure \(\mathrm{true}\) if this assignment gives a planar crescent realization; \(\mathrm{false}\) otherwise.

\Function{CheckAssignment}{$M,s$}
\If{some \(s_k\leq0\), or two of the \(s_k\) are equal}
\State \Return \(\mathrm{false}\).
\EndIf

\State Form \(Q(s)\).

\ForAll{three-element subsets \(T\subseteq\{1,\ldots,n\}\)}
\If{\(C_T(s)\geq0\)}
\State \Return \(\mathrm{false}\).
\EndIf
\EndFor

\State \Comment{In particular, \(\{1,2,3\}\) is a nondegenerate reference triangle.}

\For{\(r=4,\ldots,n\)}
\If{\(C_{\{1,2,3,r\}}(s)\neq0\)}
\State \Return \(\mathrm{false}\).
\EndIf
\EndFor

\ForAll{pairs \(i,j\) with \(4\leq i<j\leq n\)}
\If{\(C_{\{1,2,3,i,j\}}(s)\neq0\)}
\State \Return \(\mathrm{false}\).
\EndIf
\EndFor

\State \Comment{The reference-triangle conditions establish planar realizability.}

\ForAll{four-element subsets \(T\subseteq\{1,\ldots,n\}\)}
\If{\(E_T(s)=0\)}
\State \Return \(\mathrm{false}\).
\EndIf
\EndFor

\State \Return \(\mathrm{true}\).
\EndFunction
\end{algorithmic}
\end{algorithm}

The reference-triangle conditions are justified by Theorem~\ref{thm: cayley general}. The inequalities \(C_T(s)<0\) for triples require nondegenerate triangles. After planar realizability has been established, the final test excludes four concyclic points. The loop over pairs \(4\leq i<j\leq n\) is empty when \(n=4\).

Every test in Algorithm~\ref{alg: check4} is exact. An algebraic assignment specifies an irreducible polynomial \(g\in\mathbb Q[t]\), one of its real roots \(\alpha\), and polynomials \(r_k\in\mathbb Q[t]\) such that \(s_k=r_k(\alpha)\). To test whether \(P(\alpha)=0\), divide \(P\) by \(g\) and check whether the remainder is zero. A nonzero remainder proves \(P(\alpha)\neq0\), because \(g\) is irreducible. Signs are determined by exact comparison of real algebraic numbers, not by rounding to a fixed number of decimal places.

We applied Algorithm~\ref{alg: check4} to the exact algebraic assignments obtained in our computation and retained each accepted representative once. The resulting list \(\mathcal R\) contains \(31\) five-point representatives. The exact assignments are available as described in Remark~\ref{rem:computational-materials}.

A return value of \(\mathrm{false}\) concerns only the input assignment. It does not rule out a different assignment for the same representative.

\subsection{Exact exclusion of the remaining five-point classes}

The final procedure considers every representative in \(\mathcal S\) that has not been included in \(\mathcal R\). It uses the exclusion argument from Section~\ref{sec:five-point}. Its negative conclusion concerns all possible distance assignments for the class.

\begin{algorithm}[htbp]
\caption{Exact exclusion of the remaining five-point classes}
\label{alg:exclude5}
\small
\begin{algorithmic}
\Require The list \(\mathcal S\) of five-point strict representatives and the sublist \(\mathcal R\) with verified realizations.
\Ensure Lists \(\mathcal X\) of excluded classes and \(\mathcal U\) of classes not decided by this test.

\State \(\mathcal X\gets\) empty list; \(\mathcal U\gets\) empty list.
\State Use variables \(u,s_2,s_3,s_4\) and set \(s_1=1\).

\ForAll{\(M\) in \(\mathcal S\) but not in \(\mathcal R\)}
\State Form \(Q(s)\) from \(M\).

\State \(F_4\gets C_{\{1,2,3,4\}}(s)\), \(F_5\gets C_{\{1,2,3,5\}}(s)\).
\State \(F_{45}\gets C_{\{1,2,3,4,5\}}(s)\).

\State \(H\gets s_2s_3s_4\).
\ForAll{\(1\leq i<j\leq4\)}
\State \(H\gets H(s_i-s_j)\).
\EndFor
\ForAll{three-element subsets \(T\subseteq\{1,2,3,4,5\}\)}
\State \(H\gets H C_T(s)\).
\EndFor

\State Compute a Gr\"obner basis \(G\) of \(\langle F_4,F_5,F_{45},uH-1\rangle\) over \(\mathbb Q\).
\Statex \hspace{\algorithmicindent} Use variable order \(u,s_2,s_3,s_4\) and \texttt{order="grevlex"}. \Comment{We use SymPy Package in Python here.}

\If{\(G\) consists only of the constant polynomial \(1\)}
\State Append \(M\) to \(\mathcal X\).
\Else
\State Append \(M\) to \(\mathcal U\).
\EndIf
\EndFor

\State \Return \((\mathcal X,\mathcal U)\).
\end{algorithmic}
\end{algorithm}

The polynomial \(H\) is the product of the squared distances, their pairwise differences, and the ten three-point Cayley--Menger determinants. Every planar crescent realization would therefore satisfy
\[
F_4=F_5=F_{45}=0,\qquad H\neq0.
\]
Setting \(u=1/H\) would make all four generators in Algorithm~\ref{alg:exclude5} vanish. An output \(G=\{1\}\) contradicts this by Lemma~\ref{lem:gb-nonvanishing}.

No conclusion about realizability is drawn when this test does not return \(\{1\}\). Such a class is placed in \(\mathcal U\), rather than being declared realizable.

For the five-point computation, the outputs are
\[
|\mathcal S|=49,\qquad
|\mathcal R|=31,\qquad
|\mathcal X|=18,\qquad
|\mathcal U|=0.
\]
The realization and exclusion lists are disjoint and together contain every strict representative. Thus the result is a complete classification of the strict distance classes, not an exhaustive listing of their individual geometric realizations.
\pagebreak

\section{Representatives and Realization Data}\label{app: 5points}

The four-point examples are shown in Figure~\ref{fig:mcrapp}. For five points, Tables~\ref{tab: matrixTABLE} and~\ref{tab: distanceTABLE} list the \(31\) strict classes admitting planar crescent realizations.

The class numbers are their positions in the sorted list of all \(49\) strict representatives. They are not the numbering of the illustrative configurations in Figure~\ref{fig:allFIVE}, nor the order in which certificates occur in the verification program. A certificate is matched to its class by its complete upper-triangular label sequence.

In Table~\ref{tab: matrixTABLE}, the sequence
\[
(a_1,\ldots,a_{10})
\]
represents
\[
(m_{12},m_{13},m_{14},m_{15},m_{23},
 m_{24},m_{25},m_{34},m_{35},m_{45})
=
(d_{a_1},\ldots,d_{a_{10}}).
\]
The lower-triangular entries are determined by symmetry and the diagonal entries are zero.
\begin{figure}[h]
\centering
\resizebox{\linewidth}{!}{%
\begin{tikzpicture}[
    x=1cm,y=1cm,
    line width=0.55pt,
    line cap=round,
    line join=round,
    font=\small
]

\begin{scope}

    \begin{scope}[shift={(-1.10,{0.6*sqrt(8)})},scale=0.6]
        \coordinate (A) at (0,{sqrt(3)});
        \coordinate (B) at (0,{-sqrt(8)});
        \coordinate (C) at (-1,0);
        \coordinate (D) at (1,0);

        \draw[green!50!black] (A)--(B);
        \draw[blue] (B)--(C) (B)--(D);
        \draw[red] (A)--(C)--(D)--cycle;

        \foreach \P in {A,B,C,D}
            \node[circle,fill=black,inner sep=1.05pt] at (\P) {};

        \node[above,inner sep=2pt,font=\scriptsize] at (A) {$p_1$};
        \node[below,inner sep=2pt,font=\scriptsize] at (B) {$p_2$};
        \node[left,inner sep=2pt,font=\scriptsize] at (C) {$p_3$};
        \node[right,inner sep=2pt,font=\scriptsize] at (D) {$p_4$};
    \end{scope}

    \begin{scope}[shift={(1.10,{0.6*sqrt(8)})},scale=0.6]
        \coordinate (A) at (0,{-sqrt(3)});
        \coordinate (B) at (0,{-sqrt(8)});
        \coordinate (C) at (-1,0);
        \coordinate (D) at (1,0);

        \draw[green!50!black] (A)--(B);
        \draw[blue] (B)--(C) (B)--(D);
        \draw[red] (A)--(C)--(D)--cycle;

        \foreach \P in {A,B,C,D}
            \node[circle,fill=black,inner sep=1.05pt] at (\P) {};

        \node[right,inner sep=2pt,font=\scriptsize] at (A) {$p_1$};
        \node[below,inner sep=2pt,font=\scriptsize] at (B) {$p_2$};
        \node[left,inner sep=2pt,font=\scriptsize] at (C) {$p_3$};
        \node[right,inner sep=2pt,font=\scriptsize] at (D) {$p_4$};
    \end{scope}

    \node at (-1.10,-0.55) {$\mathrm M_+$};
    \node at (1.10,-0.55) {$\mathrm M_-$};

    \node at (0,-1.75) {$
        M_{\mathrm M}=
        \begin{pmatrix}
            0   & d_1 & d_3 & d_3\\
            d_1 & 0   & d_2 & d_2\\
            d_3 & d_2 & 0   & d_3\\
            d_3 & d_2 & d_3 & 0
        \end{pmatrix}
    $};

    \node[anchor=north] at (0,-2.85) {$
        \begin{aligned}
            \boldsymbol{\lambda}_+
                &=\bigl(\sqrt8+\sqrt3,\,3,\,2\bigr),\\
            \boldsymbol{\lambda}_-
                &=\bigl(\sqrt8-\sqrt3,\,3,\,2\bigr).
        \end{aligned}
    $};

\end{scope}

\begin{scope}[xshift=5.1cm]

    \begin{scope}[scale=1.35]
        \coordinate (A) at (-0.5,0);
        \coordinate (B) at (0.5,0);
        \coordinate (C) at (0,2);
        \coordinate (D) at (0,{15/16});

        \draw[green!50!black] (A)--(B);
        \draw[blue] (A)--(C) (B)--(C);
        \draw[red] (A)--(D) (B)--(D) (C)--(D);

        \foreach \P in {A,B,C,D}
            \node[circle,fill=black,inner sep=1.05pt] at (\P) {};

        \node[below left,inner sep=2pt,font=\scriptsize]
            at (A) {$p_1$};
        \node[below right,inner sep=2pt,font=\scriptsize]
            at (B) {$p_2$};
        \node[above,inner sep=2pt,font=\scriptsize]
            at (C) {$p_3$};
        \node[right,inner sep=2pt,font=\scriptsize]
            at (D) {$p_4$};
    \end{scope}

    \node at (0,-0.55) {$\mathrm C$};

    \node at (0,-1.75) {$
        M_{\mathrm C}=
        \begin{pmatrix}
            0   & d_1 & d_2 & d_3\\
            d_1 & 0   & d_2 & d_3\\
            d_2 & d_2 & 0   & d_3\\
            d_3 & d_3 & d_3 & 0
        \end{pmatrix}
    $};

    \node[anchor=north] at (0,-2.85) {$
        \boldsymbol{\lambda}
        =\left(1,\,\frac{\sqrt{17}}2,\,\frac{17}{16}\right).
    $};

\end{scope}

\begin{scope}[xshift=10.2cm]

    \begin{scope}[shift={({-0.45*sqrt(7)},1.35)},scale=1.8]
        \coordinate (A) at (0,0);
        \coordinate (B) at ({(sqrt(7)-2)/2},-0.5);
        \coordinate (C) at ({sqrt(7)/2},0.5);
        \coordinate (D) at ({sqrt(7)/2},-0.5);

        \draw[green!50!black] (A)--(B);
        \draw[blue] (B)--(D) (C)--(D);
        \draw[red] (A)--(C) (A)--(D) (B)--(C);

        \foreach \P in {A,B,C,D}
            \node[circle,fill=black,inner sep=1.05pt] at (\P) {};

        \node[left,inner sep=2pt,font=\scriptsize] at (A) {$p_1$};
        \node[below,inner sep=2pt,font=\scriptsize] at (B) {$p_2$};
        \node[above,inner sep=2pt,font=\scriptsize] at (C) {$p_3$};
        \node[below,inner sep=2pt,font=\scriptsize] at (D) {$p_4$};
    \end{scope}

    \node at (0,-0.55) {$\mathrm R$};

    \node at (0,-1.75) {$
        M_{\mathrm R}=
        \begin{pmatrix}
            0   & d_1 & d_3 & d_3\\
            d_1 & 0   & d_3 & d_2\\
            d_3 & d_3 & 0   & d_2\\
            d_3 & d_2 & d_2 & 0
        \end{pmatrix}
    $};

    \node[anchor=north] at (0,-2.85) {$
        \boldsymbol{\lambda}
        =\left(\sqrt{3-\sqrt7},\,1,\,\sqrt2\right).
    $};

\end{scope}

\end{tikzpicture}%
}
\caption{Examples of the four-point strict distance classes M, C, and R. Green, blue, and red edges carry labels $d_1,d_2,d_3$, respectively. Matrix rows and columns are ordered by $p_1,p_2,p_3,p_4$; $\boldsymbol{\lambda} =(\lambda_1,\lambda_2,\lambda_3)$ gives the numerical distances. Drawings are independently scaled.}
\label{fig:mcrapp}
\end{figure}

\begin{longtable}{|c|c|}
\caption{Canonical representatives of the \(31\) realizable five-point strict distance classes.}
\label{tab: matrixTABLE}\\
\hline
Class & Upper-triangular label sequence\\
\hline
\endfirsthead
\hline
Class & Upper-triangular label sequence\\
\hline
\endhead
\hline
\endfoot
\endlastfoot

1  & \((1,2,2,3,3,4,3,4,4,4)\)\\
2  & \((1,2,2,3,3,4,4,3,4,4)\)\\
5  & \((1,2,2,3,4,4,4,3,3,4)\)\\
7  & \((1,2,2,4,3,4,4,3,3,4)\)\\
10 & \((1,2,3,3,3,4,4,2,4,4)\)\\
13 & \((1,2,3,3,4,2,4,4,3,4)\)\\
14 & \((1,2,3,3,4,2,4,4,4,3)\)\\
16 & \((1,2,3,3,4,3,4,4,2,4)\)\\
18 & \((1,2,3,3,4,4,4,2,3,4)\)\\
20 & \((1,2,3,4,2,4,4,3,3,4)\)\\
21 & \((1,2,3,4,2,4,4,3,4,3)\)\\
23 & \((1,2,3,4,3,4,2,3,4,4)\)\\
24 & \((1,2,3,4,3,4,4,2,3,4)\)\\
25 & \((1,2,3,4,3,4,4,3,2,4)\)\\
26 & \((1,2,3,4,3,4,4,3,4,2)\)\\
27 & \((1,2,3,4,4,2,4,3,3,4)\)\\
28 & \((1,2,3,4,4,2,4,3,4,3)\)\\
29 & \((1,2,3,4,4,3,3,2,4,4)\)\\
31 & \((1,2,3,4,4,3,4,2,4,3)\)\\
32 & \((1,2,3,4,4,3,4,3,2,4)\)\\
34 & \((1,2,3,4,4,3,4,4,2,3)\)\\
35 & \((1,2,3,4,4,4,4,2,3,3)\)\\
36 & \((1,2,3,4,4,4,4,3,2,3)\)\\
38 & \((1,2,4,4,3,3,4,2,3,4)\)\\
39 & \((1,2,4,4,3,3,4,2,4,3)\)\\
43 & \((1,2,4,4,4,3,4,2,3,3)\)\\
45 & \((1,2,4,4,4,3,4,3,3,2)\)\\
46 & \((1,3,3,4,3,4,4,2,2,4)\)\\
47 & \((1,3,3,4,3,4,4,2,4,2)\)\\
48 & \((1,3,3,4,3,4,4,4,2,2)\)\\
49 & \((1,3,3,4,4,4,4,2,2,3)\)\\
\hline
\end{longtable}

All entries below are exact, with \(s_1=1\). In each row of Table~\ref{tab: distanceTABLE} containing \(t\), the symbol \(t\) denotes the unique real root of the polynomial in the corresponding row of Table~\ref{tab:algebraic-parameters} within the stated open interval. The parameter is local to that row and is not a common value shared by all classes. The intervals have rational endpoints and specify the roots exactly. The numerical distances are \(\lambda_k=\sqrt{s_k}\).

\begingroup
\small
\setlength{\tabcolsep}{5pt} \renewcommand{\arraystretch}{1.35}

\begin{longtable}{|c|c|c|}
\caption{Exact specification of the algebraic parameters used in Table~\ref{tab: distanceTABLE}. Each listed polynomial has exactly one real root in the indicated interval.}
\label{tab:algebraic-parameters}\\
\hline
Class & Polynomial vanishing at \(t\) & Interval containing \(t\)\\
\hline
\endfirsthead
\hline
Class & Polynomial vanishing at \(t\) & Interval containing \(t\)\\
\hline
\endhead
\hline
\endfoot
\endlastfoot

10
& \(t^4-7t^3+14t^2-8t+1\)
& \(\left(\frac{17}{100},\frac{18}{100}\right)\)\\

16
& \(4t^3-4t^2-5t+2\)
& \(\left(\frac{33}{100},\frac{34}{100}\right)\)\\

18
& \(4t^6-19t^5+20t^4-3t^2-2t+1\)
& \(\left(\frac{58}{100},\frac{59}{100}\right)\)\\

25, 46
& \(t^4-8t^3+14t^2-7t+1\)
& \(\left(\frac{25}{100},\frac{26}{100}\right)\)\\

28
& \(t^3-9t^2+6t-1\)
& \(\left(\frac{28}{100},\frac{29}{100}\right)\)\\

29
& \(4t^3-11t^2+7t-1\)
& \(\left(\frac{65}{100},\frac{66}{100}\right)\)\\

31
& \(t^4-24t^3+26t^2-9t+1\)
& \(\left(\frac{29}{100},\frac{30}{100}\right)\)\\

32
& \(t^5-12t^4+4t^3+9t^2-6t+1\)
& \(\left(\frac{63}{100},\frac{64}{100}\right)\)\\

34
& \(t^4-36t^3+33t^2-10t+1\)
& \(\left(\frac{40}{100},\frac{41}{100}\right)\)\\

35
& \(4t^4-25t^3+35t^2-15t+2\)
& \(\left(\frac{447}{100},\frac{448}{100}\right)\)\\

36, 49
& \(4t^3-t^2-3t+1\)
& \(\left(\frac{34}{100},\frac{35}{100}\right)\)\\

38
& \(4t^6-29t^5+71t^4-75t^3+39t^2-10t+1\)
& \(\left(\frac{27}{100},\frac{28}{100}\right)\)\\

39
& \(4t^3-t^2-3t+1\)
& \(\left(\frac{79}{100},\frac{80}{100}\right)\)\\

48
& \(t^4-4t^3+9t^2-6t+1\)
& \(\left(\frac{25}{100},\frac{26}{100}\right)\)\\

\hline
\end{longtable}
\endgroup

\begingroup
\small
\setlength{\tabcolsep}{5pt} \renewcommand{\arraystretch}{1.65}

\begin{longtable}{|c|c|c|c|}
\caption{Exact squared distances for one realization of each class in Table~\ref{tab: matrixTABLE}, normalized by \(s_1=1\). The row-dependent algebraic parameter \(t\) is specified in Table~\ref{tab:algebraic-parameters}.}
\label{tab: distanceTABLE}\\
\hline
Class & \(s_2\) & \(s_3\) & \(s_4\)\\
\hline
\endfirsthead
\hline
Class & \(s_2\) & \(s_3\) & \(s_4\)\\
\hline
\endhead
\hline
\endfoot
\endlastfoot

1
& \(\frac{3-\sqrt3}{2}\)
& \(\frac{2+\sqrt2}{2}\)
& \(\frac{(2-\sqrt3)(2+\sqrt2)}{2}\)\\[3pt]

2
& \(\frac{4-\sqrt3}{2}\)
& \(\frac{2-\sqrt3}{2}\)
& \(\frac12\)\\[3pt]

5
& \(\frac{1+\sqrt3}{2}\)
& \(\frac{2+\sqrt3-\sqrt{3+2\sqrt3}}{2}\)
& \(\frac{1-\sqrt{2\sqrt3-3}}{2}\)\\[3pt]

7
& \(\frac12\)
& \(\frac{3-\sqrt6}{6}\)
& \(\frac{3-\sqrt6}{2}\)\\[3pt]

10
& \(t^2(4-t)\)
& \(-t^3+5t^2-4t+1\)
& \(t\)\\[3pt]

13
& \(\frac{13+\sqrt{73}}{8}\)
& \(\frac{23+3\sqrt{73}}{8}\)
& \(\frac{9+\sqrt{73}}{8}\)\\[3pt]

14
& \(\frac{27}{43}\)
& \(\frac{16}{43}\)
& \(\frac{7}{43}\)\\[3pt]

16
& \(2-t\)
& \(t\)
& \(\frac12\)\\[3pt]

18
& \(\frac{
    \begin{gathered}
    -180t^5+783t^4-610t^3\\[-2pt]
    -157t^2+49t+69
    \end{gathered}
    }{29}\)
& \(\frac{
    \begin{gathered}
    -52t^5+203t^4-115t^3\\[-2pt]
    +12t^2-20t+18
    \end{gathered}
    }{29}\)
& \(t\)\\[3pt]

20
& \(2-\sqrt3\)
& \(\frac{2-\sqrt3}{2}\)
& \(\frac12\)\\[3pt]

21
& \(\frac{\sqrt{13}-3}{2}\)
& \(\frac{4-\sqrt{13}}{3}\)
& \(\frac{\sqrt{13}-1}{6}\)\\[3pt]

23
& \(\frac{2\sqrt{13}-1}{17}\)
& \(\frac{7+3\sqrt{13}}{34}\)
& \(\frac{9-\sqrt{13}}{34}\)\\[3pt]

24
& \(\frac{1+\sqrt3}{2}\)
& \(\frac{2+\sqrt3-\sqrt{3+2\sqrt3}}{2}\)
& \(\frac{1}{2+\sqrt{2\sqrt3}}\)\\[3pt]

25
& \(t(-t^2+6t-1)\)
& \(-t^3+7t^2-7t+2\)
& \(t\)\\[3pt]

26
& \(8-3\sqrt7\)
& \(6-2\sqrt7\)
& \(3-\sqrt7\)\\[3pt]

27
& \(4+\sqrt{13}\)
& \(\frac{3+\sqrt{13}}{2}\)
& \(\frac{11+3\sqrt{13}}{2}\)\\[3pt]

28
& \(-t^2+9t-2\)
& \(-3t^2+26t-7\)
& \(t\)\\[3pt]

29
& \(4t^2-4t+1\)
& \(1-t\)
& \(t\)\\[3pt]

31
& \(3t^3-71t^2+55t-9\)
& \(-4t^3+95t^2-80t+16\)
& \(t\)\\[3pt]

32
& \(\begin{gathered}
    5t^4-58t^3-3t^2\\[-2pt]
    +43t-12
    \end{gathered}\)
& \(-3t^4+35t^3-27t+9\)
& \(t\)\\[3pt]

34
& \(-3t^3+107t^2-62t+9\)
& \(-4t^3+143t^2-96t+16\)
& \(t\)\\[3pt]

35
& \(12t^3-71t^2+83t-18\)
& \(4t^3-25t^2+34t-8\)
& \(t\)\\[3pt]

36
& \(1-t\)
& \(4t^2+2t-1\)
& \(t\)\\[3pt]

38
& \(\begin{gathered}
    16t^5-100t^4+180t^3\\[-2pt]
    -95t^2+16t+1
    \end{gathered}\)
& \(\begin{gathered}
    12t^5-79t^4+159t^3\\[-2pt]
    -112t^2+33t-3
    \end{gathered}\)
& \(t\)\\[3pt]

39
& \(4t^2-4t+1\)
& \(1-t\)
& \(t\)\\[3pt]

43
& \(2\)
& \(6-2\sqrt7\)
& \(3-\sqrt7\)\\[3pt]

45
& \(8+3\sqrt7\)
& \(3+\sqrt7\)
& \(6+2\sqrt7\)\\[3pt]

46
& \(t^3-7t^2+10t-2\)
& \(t\)
& \(2t^3-15t^2+21t-4\)\\[3pt]

47
& \(\frac{4-\sqrt7}{9}\)
& \(\frac{5+\sqrt7}{9}\)
& \(\frac{8-2\sqrt7}{9}\)\\[3pt]

48
& \(t^2\)
& \(t^3-3t^2+6t-1\)
& \(t\)\\[3pt]

49
& \(3t-1\)
& \(4t^2+2t-1\)
& \(t\)\\[3pt]

\hline
\end{longtable}
\endgroup

\bibliographystyle{plain}
\bibliography{bib}
\end{document}